\documentclass[11pt]{article}

\usepackage{fullpage}
\usepackage{amsmath,amssymb,amsthm,hyperref,bm}
\usepackage[noadjust]{cite}
\usepackage{tikz, tikz-cd}
\usetikzlibrary{positioning}
\usepackage{mathtools}
\usepackage{verbatim}
\usepackage{multirow}
\usepackage{array}
\usepackage{stmaryrd}
\usepackage{enumitem} 
\usepackage[english]{babel}

\hypersetup{hidelinks} 

\makeatletter
\renewcommand*\env@matrix[1][\arraystretch]{%
  \edef\arraystretch{#1}%
  \hskip -\arraycolsep
  \let\@ifnextchar\new@ifnextchar
  \array{*\c@MaxMatrixCols c}}
\makeatother

\theoremstyle{definition}
\newtheorem{definition}{Definition}[section]
\newtheorem{remark}[definition]{Remark}

\newtheorem{example}[definition]{Example}

\newtheorem{problem}[definition]{Problem}
\newtheorem{notation}[definition]{Notation}

\theoremstyle{plain}
\newtheorem{theorem}[definition]{Theorem}
\newtheorem{lemma}[definition]{Lemma}
\newtheorem{corollary}[definition]{Corollary}
\newtheorem{proposition}[definition]{Proposition}

\renewcommand\arraystretch{1} 
\setlist[itemize]{itemsep=0em}

\def\tet{\boxtimes}
\def\<{\langle}
\def\>{\rangle}
\def\Span{\operatorname{Span}}
\def\L{L(\mathfrak{sl}_2)^+}
\def\sl2{\mathfrak{sl}_2}

\begin{document}

\title{\bf Four bases for the Onsager Lie algebra \\ related by a $\mathbb{Z}_2 \times \mathbb{Z}_2$ action}

\author {
Jae-Ho Lee\thanks{Department of Mathematics and Statistics, University of North Florida, Jacksonville, FL 32224, U.S.A.
}\\ \\
\emph{\large Dedicated to Professor Paul Terwilliger on the occasion of his 70th birthday}  \\ \
}
\date{}

\maketitle
\newcommand\blfootnote[1]{%
\begingroup
\renewcommand\thefootnote{}\footnote{#1}%
\addtocounter{footnote}{-1}%
\endgroup}
\blfootnote{E-mail address:  \texttt{jaeho.lee@unf.edu}}

\vspace{-1cm}
\begin{abstract}
The Onsager Lie algebra $O$ is an infinite-dimensional Lie algebra defined by generators $A$, $B$ and relations $[A, [A, [A, B]]] = 4[A, B]$ and $[B, [B, [B, A]]] = 4[B, A]$.
Using an embedding of $O$ into the tetrahedron Lie algebra $\tet$, we obtain four direct sum decompositions of the vector space $O$, each consisting of three summands.
As we will show, there is a natural action of $\mathbb{Z}_2 \times \mathbb{Z}_2$ on these decompositions. 
For each decomposition, we provide a basis for each summand.
Moreover, we describe the Lie bracket action on these bases and show how they are recursively constructed from the generators $A$, $B$ of $O$.
Finally, we discuss the action of $\mathbb{Z}_2 \times \mathbb{Z}_2$ on these bases and determine some transition matrices among the bases.

\bigskip
\noindent
\textbf{Keywords:} tetrahedron Lie algebra; Onsager Lie algebra; three-point $\mathfrak{sl}_2$-loop algebra

\smallskip
\noindent 
\textbf{2020 Mathematics Subject Classification:} 17B05, 17B65
\end{abstract}
\section{Introduction}

This paper builds on our previous work \cite{2024LeeJA} concerning the Onsager Lie algebra.
This Lie algebra was introduced by Lars Onsager \cite{1944Ons} and has made a seminal contribution to the theory of exactly solved models in statistical mechanics.
The Onsager Lie algebra can be described in several ways.
For example, in \cite{1944Ons, 2000DR} a linear basis is given together with the Lie bracket action on this basis.
In \cite{1989Perk}, Perk gave a presentation of the Onsager Lie algebra by generators and relations.
We now recall this presentation.
Throughout this paper, let $\mathbb{F}$ denote a field of characteristic zero.

\begin{definition}[\cite{1944Ons,1989Perk}]\label{def:Onsager}
Let ${O}$ denote the Lie algebra over $\mathbb{F}$ with generators $A, B$ and relations
\begin{align}
	[A, [A, [A, B]]] & = 4[A,B], \label{OnsagerDG(1)}\\
	[B, [B, [B, A]]] & = 4[B,A]. \label{OnsagerDG(2)}
\end{align}
We call ${O}$ the \emph{Onsager Lie algebra}. 
We call $A, B$ the \emph{standard generators} for ${O}$.
The relations \eqref{OnsagerDG(1)}, \eqref{OnsagerDG(2)} are the \emph{Dolan-Grady relations} \cite{1982DG}.
\end{definition}

The Onsager Lie algebra $O$ is closely related to a Lie algebra $\tet$, called the tetrahedron algebra. 
In \cite{2007HarTer}, Hartwig and Terwilliger defined $\tet$ by generators and relations.
They used $\tet$ to obtain a presentation of the three-point $\mathfrak{sl}_2$ loop algebra $\L$.
Specifically, they displayed a Lie algebra isomorphism $\tet \to \L$; cf. Lemma \ref{def: sigma}.
We now recall the definition of $\tet$.
For notational convenience, we define the set $\mathbb{I} = \{0, 1, 2, 3\}$.

\begin{definition}[{\cite[Definition 1.1]{2007HarTer}}]\label{Def:Tet-alg}
Let $\tet$ denote the Lie algebra over $\mathbb{F}$ that has generators
\begin{equation}\label{standard generators tet}
	\{x_{ij} \mid i,j \in \mathbb{I}, \ i \ne j\}, 
\end{equation}
and the following relations:
\begin{itemize}
	\item[(i)] For distinct $i,j \in \mathbb{I}$,
	\begin{equation*}\label{def:tet rel(1)}
		x_{ij} + x_{ji} = 0.
	\end{equation*}
	\item[(ii)] For mutually distinct $i,j,k \in \mathbb{I}$,
	\begin{equation*}
		[x_{ij}, x_{jk}] = 2x_{ij}+2x_{jk}.
	\end{equation*}
	\item[(iii)] For mutually distinct $h, i, j, k \in \mathbb{I}$,
	\begin{equation*}\label{Dolan-Grady rels}
		[x_{ij}, [ x_{ij}, [x_{ij}, x_{kh}]]] = 4[x_{ij}, x_{kh}].
	\end{equation*}
\end{itemize}
We call $\tet$ the \emph{tetrahedron algebra}.
We call the elements $x_{ij}$ in \eqref{standard generators tet} the \emph{standard generators} for $\tet$.
\end{definition}
\begin{remark}
Consider the symmetric group $S_4$ on the set $\mathbb{I}$. 
The group $S_4$ acts on the standard generators of $\tet$ by permuting the indices in $\mathbb{I}$. 
This action induces an injective group homomorphism $S_4 \to \operatorname{Aut}(\tet)$; cf. \cite[Corollary 12.7]{2007HarTer}.
This homomorphism gives an action of $S_4$ on $\tet$ as a group of automorphisms.
\end{remark}

We now recall how the Lie algebras $O$ and $\tet$ are related.
The following facts are taken from \cite{2007HarTer}.
For mutually distinct $h, i, j, k\in \mathbb{I}$, there exists an injective Lie algebra homomorphism ${O} \to \tet$ that sends
\begin{equation}\label{eq:inj Lie O to tet}
	A \ \longmapsto \ x_{ij}, \qquad \qquad 
	B \ \longmapsto \ x_{kh}.
\end{equation}
We call the image of $O$ under this homomorphism an \emph{Onsager subalgebra of $\tet$}.
The Lie algebra $\tet$ contains three Onsager subalgebras.
Moreover, the vector space $\tet$ is a direct sum of these three Onsager subalgebras; cf. \cite[Proposition 7.8, Theorem 11.6]{2007HarTer}.


In \cite{2024LeeJA}, we introduced a type of element in $\tet$ that “looks like” a standard generator.
For distinct $i,j \in \mathbb{I}$, consider a standard generator $x_{ij}$ of $\tet$.
An element $\xi \in \tet$ is called \emph{$x_{ij}$-like} whenever both (i) $\xi$ commutes with $x_{ij}$, (ii) $\xi$ and $x_{kh}$ satisfy a Dolan-Grady relation, where $k,h$ are distinct elements of $ \mathbb{I}$ besides $i,j$; cf. Definition \ref{def:xij-like}.
Let $X_{ij}$ denote the subset of $\tet$ consisting of the $x_{ij}$-like elements.
We note that $X_{ij}$ is a subspace of $\tet$. 
In \cite[Corollary 4.9]{2024LeeJA}, for mutually distinct $i,j,k \in \mathbb{I}$ we have a direct sum of vector spaces: $\tet = X_{ij} + X_{jk} + X_{ki}$.
We call this direct sum a triangular-shaped decomposition; for more details, see \cite{2024LeeJA}.

In the present paper, we give a variation on the triangular-shaped decompositions, which we call a path-shaped decomposition.
Our path-shaped decomposition is described as follows.
For mutually distinct $h, i, j, k \in \mathbb{I}$, we show that the vector space $\tet$ satisfies
\begin{equation*}\label{eq:path-shaped ds}
	\tet = X_{kh} + X_{hi} + X_{ij} \qquad {\rm(direct \ sum)}.
\end{equation*}
Consider the Onsager subalgebra of $\tet$ generated by $x_{ij}$ and $x_{kh}$. 
We identify this subalgebra with the Onsager Lie algebra $O$ via the embedding given in \eqref{eq:inj Lie O to tet}.
We show that the vector space $O$ satisfies
\begin{equation}\label{eq:O decomp}
	O = (X_{kh} \cap O) + (X_{hi}\cap O) + (X_{ij}\cap O) \qquad {\rm(direct \ sum)}.
\end{equation}

\noindent
We have some further remarks about the decomposition \eqref{eq:O decomp}.
First, let us simplify the notation.
Due to the $S_4$-symmetry on $\tet$, without loss of generality, we may assume $h=3$, $i=1$, $j=2$, $k=0$. 
This gives
\begin{equation}\label{decomp:0312}
	O = (X_{03} \cap O) + (X_{31} \cap O) + (X_{12} \cap O) \qquad {\rm(direct \ sum)}.
\end{equation}
From this point of view, the standard generators for $O$ are $A=x_{12}$ and $B=x_{03}$.
Let $G$ denote the subgroup of $S_4$ generated by $\rho = (12)(03)$ and $\tau = (12)$.
Note that $G$ is isomorphic to $\mathbb{Z}_2 \times \mathbb{Z}_2$.
As we will see in Lemma \ref{lem:action rho, tau}, $O$ is invariant under the $G$-action.
By applying $\rho$, $\tau$, $\rho\tau$ to the decomposition \eqref{decomp:0312}, we get three more direct sums:
\begin{align}
	O & = (X_{30} \cap O) + (X_{02}\cap O) + (X_{21}\cap O) \label{decomp:3021}\\
	 & = (X_{03} \cap O) + (X_{32}\cap O) + (X_{21}\cap O) \label{decomp:0321}\\
	 & = (X_{30} \cap O) + (X_{01}\cap O) + (X_{12}\cap O). \label{decomp:3012}
\end{align}

We now describe our main results.
For each of the direct sum decompositions \eqref{decomp:0312}--\eqref{decomp:3012}, we will display a basis for each summand.
To describe these bases, we use the notation in the following table.
\begin{equation}\label{eq:4bases_table}
{\renewcommand{\arraystretch}{1.3}
\begin{tabular}{c|cccc}
	$[khij]$ & $X_{kh} \cap O$ & $X_{hi} \cap O$ & $X_{ij} \cap O$ \\
	\hline
	$[0312]$ & $\{B^{\uparrow\uparrow}_{i}\}_{i \in \mathbb{N}}$ & $\{\psi^{\uparrow\uparrow}_{i+1}\}_{i \in \mathbb{N}}$ & $\{A^{\uparrow\uparrow}_{i}\}_{i \in \mathbb{N}}$ \\
	$[3021]$ & $\{B^{\downarrow\downarrow}_{i}\}_{i \in \mathbb{N}}$ & $\{\psi^{\downarrow\downarrow}_{i+1}\}_{i \in \mathbb{N}}$ & $\{A^{\downarrow\downarrow}_{i}\}_{i \in \mathbb{N}}$ \\
	$[0321]$ & $\{B^{\downarrow\uparrow}_{i}\}_{i \in \mathbb{N}}$ & $\{\psi^{\downarrow\uparrow}_{i+1}\}_{i \in \mathbb{N}}$ & $\{A^{\downarrow\uparrow}_{i}\}_{i \in \mathbb{N}}$ \\
	$[3012]$ & $\{B^{\uparrow\downarrow}_{i}\}_{i \in \mathbb{N}}$ & $\{\psi^{\uparrow\downarrow}_{i+1}\}_{i \in \mathbb{N}}$ & $\{A^{\uparrow\downarrow}_{i}\}_{i \in \mathbb{N}}$ \\
\end{tabular}}
\end{equation}
This table illustrates the following (i)--(iv):
\begin{itemize}
\setlength\itemsep{0pt}
\item[(i)] The row-index $[khij]$ is a label that describes the decomposition of $O$ as shown in \eqref{eq:O decomp}.
\item[(ii)] The four main rows correspond to the decompositions \eqref{decomp:0312}--\eqref{decomp:3012}.
\item[(iii)] For each entry of the table, the vectors displayed form a basis for the column-index summand in the row-index decomposition.
\item[(iv)] In each row, the union of the vectors displayed forms a basis for $O$.
\end{itemize}
After introducing the four bases for $O$ in the table \eqref{eq:4bases_table}, we will discuss how the Lie bracket acts on each basis and describe how they are recursively obtained from the standard generators of $O$.
We will also discuss how the group $G$ acts on these bases. 
Furthermore, we will determine some transition matrices among the four bases.

\begin{remark} 
The main results of this paper are about $O$ and $\tet$.
Earlier, we mentioned the Lie algebra isomorphism $\tet \to \L$.
Most of the calculations in this paper will be carried out in $\L$.
\end{remark}

This paper is organized as follows.
In Section \ref{sec:L(sl2)+}, we review the Lie algebra $\L$ and its connection to $\tet$.
In Section \ref{sec:Aut(L)}, we review the $S_4$-action on $\tet$ and its induced action on $\L$. 
We define the subgroup $G$ of $S_4$, which is isomorphic to $\mathbb{Z}_2 \times \mathbb{Z}_2$, and discuss the $G$-action on the Onsager Lie algebra $O$.
In Section \ref{sec:x_ij-like}, we recall the $x_{ij}$-like elements of $\tet$. 
For mutually distinct $h, i, j, k \in \mathbb{I}$, we get a direct sum of vector spaces: $\tet = X_{kh} + X_{hi} + X_{ij}$.
In Section \ref{sec:DSO}, we decompose the vector space $O$ into the direct sums given in \eqref{decomp:0312}--\eqref{decomp:3012}.
In Sections \ref{sec:0312}--\ref{sec:3012}, we construct the bases displayed in table \eqref{eq:4bases_table}.
We then describe the action of the Lie bracket on these bases and examine how they are recursively obtained from the standard generators of $O$.
In Section \ref{sec:action rho tau}, we discuss the $G$-action on the four bases.
In Section \ref{sec:TM}, we find some transition matrices among the four bases.

\medskip
\noindent
\textbf{Assumptions and Notations.}
Throughout the paper, we adopt the following assumptions and notational conventions. 
Recall the set of natural numbers $\mathbb{N} = \{0, 1, 2, \ldots\}$ and the set of positive integers $\mathbb{N}^+ = \mathbb{N} \setminus \{0\} = \{1, 2, 3, \ldots\}$.
Let $\mathbb{F}$ denote a field of characteristic zero.
All vector spaces and tensor products mentioned are understood to be over $\mathbb{F}$.
Every algebra and Lie algebra mentioned is over $\mathbb{F}$.
Every algebra mentioned without the ``Lie'' prefix is understood to be associative and have a multiplicative identity.

\section{The three-point $\mathfrak{sl}_2$ loop algebra}\label{sec:L(sl2)+}

In this section, we review the definition of the three-point $\mathfrak{sl}_2$ loop algebra and recall its connection to the Lie algebra $\tet$, as established in \cite{2007HarTer}.
We begin by recalling the Lie algebra $\mathfrak{sl}_2$, which has a basis $e, f, h$ and Lie bracket 
\begin{equation*}
	[h,e] = 2e, \qquad \qquad
	[h,f] = -2f, \qquad \qquad 
	[e,f] = h.
\end{equation*}
Consider the elements $x,y,z$ of $\mathfrak{sl}_2$ defined by
\begin{equation*}
	x=2e-h, \qquad \qquad
	y=-2f-h, \qquad \qquad 
	z=h.
\end{equation*}
By \cite[Lemma 3.2]{2007HarTer}, the elements $x,y,z$ form a basis for $\mathfrak{sl}_2$ and satisfy
\begin{equation}\label{eq:[x,y]}
	[x,y] = 2x+2y, \qquad \qquad
	[y,z] = 2y+2z, \qquad \qquad 
	[z,x] = 2z+2x.
\end{equation}
The basis $x,y,z$ is called the \emph{equitable basis} \cite{2007HarTer,2011BenTerMZ}.
Let $t$ denote an indeterminate, and let $\mathbb{F}[t, t^{-1}, (t-1)^{-1}]$ denote the algebra consisting of all Laurent polynomials in $t$ and $t-1$ with coefficients in $\mathbb{F}$.
We abbreviate $\mathcal{A} = \mathbb{F}[t, t^{-1}, (t-1)^{-1}]$.
By \cite[Lemma 6.2]{2007HarTer}, there exists a unique algebra automorphism $\prime$ of $\mathcal{A}$ that sends $t$ to $1-t^{-1}$. 
Observe that this algebra automorphism has order $3$ and satisfies $t'=1-t^{-1}$ and $t'' = (1-t)^{-1}$. 
By construction, the vector space $\mathcal{A}$ has a basis
\begin{equation}\label{basis for A}
	1, \quad t^n, \quad (t')^n, \quad (t'')^n, \qquad \qquad n \in \mathbb{N}^+.
\end{equation}

\begin{definition}[{\cite[Definition 6.1]{2007HarTer}}] \label{Def:3pt sl2-loop alg}
Let $L(\sl2)^+$ denote the Lie algebra consisting of the vector space $\sl2 \otimes \mathcal{A}$ and Lie bracket
\begin{equation}\label{def:eq[u,v]x(ab)}
	[u\otimes a, v \otimes b] = [u,v] \otimes ab, \qquad \qquad u,v \in \sl2, \qquad a, b \in \mathcal{A}. 
\end{equation}
We call $L(\sl2)^+$ the \emph{three-point $\sl2$ loop algebra}.
\end{definition}
\noindent
Note that $\L$ is a right $\mathcal{A}$-module, with the action map defined as 
\begin{equation}\label{eq:A-module L}
	\L \times \mathcal{A} \ \longrightarrow \ \L, \qquad 
	(u\otimes a , b) \ \longmapsto \ u\otimes ab.
\end{equation}
By construction, the elements
\begin{equation}\label{eq:basis L(sl2)+}
	u \otimes 1, \quad 
	u \otimes t^n, \quad 
	u \otimes (t')^n, \quad
	u \otimes (t'')^n, 
	\qquad \qquad u \in\{x,y,z\}, \quad n \in \mathbb{N}^+,
\end{equation}
form a basis for $\L$.

We now recall how the Lie algebras $\tet$ and $\L$ are related.
\begin{lemma}[{\cite[Proposition 6.5]{2007HarTer}}]\label{def: sigma}
There exists a unique Lie algebra isomorphism $\sigma: \tet \to L(\sl2)^+$ that sends
\begin{align*}
	&& x_{12} \quad & \longmapsto \quad x \otimes 1, && x_{03} \quad \longmapsto \quad  y\otimes t + z \otimes (t-1), &&\\
	&& x_{23} \quad & \longmapsto \quad y \otimes 1, && x_{01} \quad \longmapsto \quad  z\otimes t' + x \otimes (t'-1), &&\\
	&& x_{31} \quad & \longmapsto \quad z \otimes 1, && x_{02} \quad \longmapsto \quad  x\otimes t'' + y \otimes (t''-1), &&
\end{align*}
where $x,y,z$ is the equitable basis for $\sl2$.
\end{lemma}

\noindent
Recall the standard generators $x_{ij}$ of $\tet$ from \eqref{standard generators tet}.
By a \emph{standard generator of $\L$}, we mean an element $x^\sigma_{ij}$ $(i,j\in \mathbb{I}, i\neq j)$.
Recall the Onsager Lie algebra $O$ with the standard generators $A$ and $B$.
Also, recall the three Onsager subalgebras of $\tet$ from the below \eqref{eq:inj Lie O to tet}.
By an \emph{Onsager subalgebra of $\L$}, we mean the $\sigma$-image of an Onsager subalgebra of $\tet$.
For the rest of this paper, for notational convenience, we identify $O$ with one of the three Onsager subalgebras of $\L$, namely, $A=x^\sigma_{12}$ and $B=x^\sigma_{03}$.
By Lemma \ref{def: sigma},
\begin{equation}\label{eq: A,B}
	A = x\otimes 1, \qquad B = y \otimes t + z \otimes (t-1).
\end{equation}
From this point of view, $O$ is described as follows.
By \cite[Lemma 7.5]{2024LeeJA}, 
\begin{equation}\label{decomp O sigma}
	{O} = x\otimes \mathbb{F}[t] + y \otimes t \mathbb{F}[t] + z\otimes(t-1)\mathbb{F}[t] \qquad  (\text{\rm direct sum}).
\end{equation}
By \eqref{decomp O sigma}, the Onsager Lie algebra $O$ has a basis
\begin{equation}\label{eq:basis for Osigma}
	x \otimes t^n, \qquad 
	y \otimes t^{n+1}, \qquad 
	z \otimes (t-1)t^{n}, \qquad \qquad n \in \mathbb{N}.
\end{equation}

\section{Automorphisms of $\L$}\label{sec:Aut(L)}
In this section, we recall the action of the symmetric group $S_4$ on $\L$. 
We will consider the action of a certain subgroup of $S_4$ that is isomorphic to $\mathbb{Z}_2 \times \mathbb{Z}_2$.
Consider the symmetric group $S_4$ on the set $\mathbb{I}$ and its generators:
\begin{align}
	&& && & \rho = (12)(30), && \tau = (12), 	&& &&	\label{eq:rho,tau}\\
	&& && & \mu = (23)(10), && \varphi = (123). && && \label{eq:mu,varphi}
\end{align}
The group $S_4$ acts on the standard generators for $\tet$ by permuting the indices.
By \cite[Corollary 12.7]{2007HarTer}, this action yields an embedding of $S_4$ into the automorphism group $\operatorname{Aut}(\tet)$ such that for $\beta \in S_4$ 
$$
	\beta(x_{ij}) = x_{\beta(i)\beta(j)} \qquad \qquad (i,j \in \mathbb{I}, \quad i \neq j).
$$	 
The Lie algebra isomorphism $\sigma: \tet \to \L$ from Lemma \ref{def: sigma} induces an $S_4$-action on $\L$ as group automorphisms.
This action is described by Elduque in \cite{2007Eld}.
We now review this action.
\begin{lemma}[cf. {\cite[Theorem 1.4]{2007Eld}}]\label{lem:EldThm1.4}
Let $\rho$, $\tau$, $\mu$, and $\varphi$ be generators of $S_4$ as in \eqref{eq:rho,tau} and \eqref{eq:mu,varphi}.
Then each of these generators acts as an automorphism of $\L$ in the following way.
\begin{itemize}
	\item[\rm(i)] $\rho$ is the automorphism of $\L$ given by
	\begin{equation}\label{eq:action_rho L}
	\rho(u \otimes a)  = \rho(u \otimes 1)a
	\end{equation}
	for all  $u\in \sl2$ and $a \in \mathcal{A}$, where
	\begin{align}
		\rho(x\otimes 1) & = -x\otimes 1,  \\
		\rho(y\otimes 1) & = (x\otimes 1 + z\otimes (1-t))t^{-1}, \\
		\rho(z\otimes 1) & = (x\otimes 1 + y\otimes t)(1-t)^{-1}. \label{eq:action_rho(z)}
	\end{align}

	\item[\rm(ii)] $\tau = \tau_{\sl2} \otimes \tau_{\mathcal{A}}$ is the automorphism of $\L$ given by
	\begin{equation*}
	\tau(u \otimes a) = \tau_{\sl2}(a) \otimes \tau_{\mathcal{A}}(a)
	\end{equation*}
	for all  $u\in \sl2$ and $a \in \mathcal{A}$, where $\tau_{\sl2}$ is the order $2$ automorphism of $\sl2$ given by
	\begin{equation*}
		\tau_{\sl2}(x) = -x, \qquad \tau_{\sl2}(y) = -z, \qquad \tau_{\sl2}(z) = -y,
	\end{equation*}
	and $\tau_{\mathcal{A}}$ is the order $2$ automorphism of $\mathcal{A}$ determined by $\tau_{\mathcal{A}}(t)=1-t$.
	
	\item[\rm(iii)] $\mu$ is the automorphism of $\L$ given by
	\begin{equation*}\label{eq:action_mu L}
	\mu(u \otimes a)  = \mu(u \otimes 1)a
	\end{equation*}
	for all  $u\in \sl2$ and $a \in \mathcal{A}$, where
	\begin{align*}
		\mu(x\otimes 1) & = y \otimes t + z \otimes (t-1),  \\
		\mu(y\otimes 1) & = -y \otimes 1, \\
		\mu(z\otimes 1) & = (x\otimes 1 + y \otimes t)(t-1)^{-1}. 
	\end{align*}

	\item[\rm(iv)] $\varphi = \varphi_{\sl2} \otimes \varphi_{\mathcal{A}}$ is the automorphism of $\L$ given by
	\begin{equation*}
	\varphi(u \otimes a) = \varphi_{\sl2}(a) \otimes \varphi_{\mathcal{A}}(a)
	\end{equation*}
	for all  $u\in \sl2$ and $a \in \mathcal{A}$, where $\varphi_{\sl2}$ is the order $3$ automorphism of $\sl2$ given by
	$$
		\varphi_{\sl2}(x) = y, \qquad \varphi_{\sl2}(y) = z, \qquad \varphi_{\sl2}(z) = x,
	$$
	and $\varphi_{\mathcal{A}}$ is the order $3$ automorphism of $\mathcal{A}$ determined by $\varphi_{\mathcal{A}}(t)=1-t^{-1}$.
	Note that $\varphi_{\mathcal{A}}$ is the same as the automorphism $\prime$ in the above line \eqref{basis for A}.
\end{itemize}
\end{lemma}

\begin{lemma}[cf. {\cite[Theorem 1.4]{2007Eld}}]\label{lem:EldThm1.4(2)}
Recall the standard generators $x_{ij}$ in $\tet$ from \eqref{standard generators tet}. 
Recall the Lie algebra isomorphism $\sigma: \tet \to \L$ from Lemma \ref{def: sigma}.
For each $\beta \in S_4$, we have
\begin{equation*}
	\beta(x^\sigma_{ij}) = x^\sigma_{\beta(i)\beta(j)} \qquad \qquad (i,j \in \mathbb{I}, \quad i \ne j).
\end{equation*}
In other words, the following diagram commutes:
\begin{equation*}\label{eq:sigma beta diagram}
    \begin{tikzcd}
    {\boxtimes} \arrow[r, rightarrow, "\sigma"] \arrow[d, rightarrow, "\beta" swap] & {\L} \arrow[d, rightarrow, "\beta"] \\
    {\boxtimes} \arrow[r, rightarrow, "\sigma"] & {\L} 
    \end{tikzcd}
\end{equation*}
\end{lemma}
 
 \begin{remark}\label{rmk:AmodL}
Recall that $\L$ is a right $\mathcal{A}$-module as defined in \eqref{eq:A-module L}.
By Lemma \ref{lem:EldThm1.4}, the automorphisms $\rho, \tau, \mu, \varphi$ on $\L$ satisfy
\begin{align*}
	&& & \rho(vb) = \rho(v)b, && \tau(vb) = \tau(v)\tau_{\mathcal{A}}(b), &&\\
	&& & \mu(vb) = \mu(v)b, && \varphi(vb) = \varphi(v)\varphi_{\mathcal{A}}(b), &&
\end{align*}
for all  $v\in \L$ and $b \in \mathcal{A}$.
\end{remark}	

Recall the Onsager subalgebra $O$ of $\L$ from \eqref{decomp O sigma} and the generators $\rho$, $\tau$, $\mu$, $\varphi$ of $S_4$ from \eqref{eq:rho,tau} and \eqref{eq:mu,varphi}.
We consider the action of these generators on $O$.
The action of $\varphi$ on $O$ has already been described in the earlier paper \cite{2024LeeJA}.
The action of $\mu$ fixes $O$ and interchanges $A$ and $B$.
We now discuss the actions of $\rho$ and $\tau$ on $O$.
To facilitate our discussion, we need the following definition.

\begin{definition}\label{def:Z2xZ2-group}
Let $G$ denote the subgroup of $S_4$ generated by $\rho$ and $\tau$.
Observe that each of $\rho$, $\tau$ is an involution. 
Moreover, $\rho$ and $\tau$ commute. 
We note that $G$ is isomorphic to $\mathbb{Z}_2 \times \mathbb{Z}_2$. 
\end{definition}

\begin{lemma}\label{lem:action rho, tau}
Recall the automorphisms $\rho,$ $\tau$ of $\L$ from Lemma \ref{lem:EldThm1.4}.
Recall the standard generators $A$, $B$ of $O$.
Then $\rho$, $\tau$ act on $A$, $B$ as follows:
\begin{align}
	&& & \rho(A) = - A, && \rho(B) = -B, && \label{eq:action rho}\\
	&& & \tau(A) =  A, && \tau(B) = -B. && \label{eq:action tau}
\end{align}
Moreover, $O$ is invariant under the $G$-action.
\end{lemma}
\begin{proof}
By \eqref{eq: A,B} and Lemma \ref{lem:EldThm1.4}(i) we have
\begin{align*}
	\rho(A) & = \rho(x\otimes1) = -x\otimes1 = -A, \\
	\rho(B) & = \rho(y\otimes1)t + \rho(z \otimes1) (t-1) = -y\otimes t - z \otimes (t-1) = -B.
\end{align*}
Thus, we obtain \eqref{eq:action rho}.
Similarly, use Lemma \ref{lem:EldThm1.4}(ii) to obtain \eqref{eq:action tau}.
For the second assertion, by \eqref{eq:action rho}, \eqref{eq:action tau} and the fact that $A, B$ generate $O$, it follows that each of $\rho$ and $\tau$ fixes $O$.
Since $\rho$ and $\tau$ generate the group $G$, the result follows.
\end{proof}

\section{The $x_{ij}$-like elements}\label{sec:x_ij-like}
Pick mutually distinct $h, i, j, k \in \mathbb{I}$. 
Consider the standard generator $x_{ij}$ in $\tet$.
In this section, we recall the $x_{ij}$-like elements of $\tet$ and then decompose $\tet$ into a direct sum of three subspaces consisting of $x_{ij}$-like, $x_{jk}$-like, and $x_{kh}$-like elements.
First, we recall what it means to be $x_{ij}$-like.

\begin{definition}[{\cite[Definition 4.1]{2024LeeJA}}]\label{def:xij-like}
Recall the Lie algebra $\tet$ and its standard generators $x_{ij}$ from Definition \ref{Def:Tet-alg}.
By an \emph{$x_{ij}$-like element in $\tet$}, we mean an element $\xi$ in $\tet$ that satisfies the following conditions:
\begin{align*}
	[x_{ij}, \xi] & = 0,\\ 
	[x_{kh}, [x_{kh}, [x_{kh}, \xi]]] & = 4[x_{kh}, \xi], 
\end{align*}
where $h, i ,j ,k \in \mathbb{I}$ are mutually distinct.
\end{definition}

For each standard generator $x_{ij}$ of $\tet$, let $X_{ij}$ denote the subset of $\tet$ consisting of the $x_{ij}$-like elements in $\tet$.
We note that $X_{ij}$ is a subspace of $\tet$ and $X_{ij}=X_{ji}$.
Recall the Lie algebra isomorphism $\sigma: \tet \to L(\sl2)$ from Lemma \ref{def: sigma}.
As we investigate $X_{ij}$, it is convenient to work with its $\sigma$-image $X^{\sigma}_{ij}$.

\begin{lemma}[{\cite[Lemma 4.5]{2024LeeJA}}]\label{lem:x_ij-elts L(sl2)+}
Recall $\mathcal{A}=\mathbb{F}[t, t^{-1}, (t-1)^{-1}]$ and recall the Lie algebra isomorphism $\sigma: \tet \to L(\sl2)$.
We have
\begin{align*}
	&& X^\sigma_{12} & = x \otimes \mathcal{A}, && X^\sigma_{03} = (y\otimes t + z \otimes (t-1))\mathcal{A}, && \\
	&& X^\sigma_{23} & = y \otimes \mathcal{A}, && X^\sigma_{01} = (z\otimes t' + x \otimes (t'-1))\mathcal{A}, && \\
	&& X^\sigma_{31} & = z \otimes \mathcal{A}, && X^\sigma_{02} = (x\otimes t'' + y \otimes (t''-1))\mathcal{A}. && 
\end{align*}
\end{lemma}

\begin{lemma}[{cf. \cite[Corollary 4.6, Lemma 4.7]{2024LeeJA}}]\label{lem:basis X_ij}
For a standard generator $x^\sigma_{ij}$ of $\L$, we have
\begin{equation}\label{eq:x_ijA}
	X^\sigma_{ij} = x^\sigma_{ij}\mathcal{A}.
\end{equation}
Moreover, the elements
\begin{equation*}
	x^\sigma_{ij}, \qquad 
	x^\sigma_{ij}t^n, \qquad 
	x^\sigma_{ij}(t')^n, \qquad 
	x^\sigma_{ij}(t'')^n, \qquad \qquad n \in \mathbb{N}^+,
\end{equation*}
form a basis for $X^\sigma_{ij}$.
\end{lemma}
\begin{proof}
The first assertion follows from Lemmas \ref{def: sigma} and \ref{lem:x_ij-elts L(sl2)+}.
For the second assertion, use \eqref{eq:x_ijA} and the fact that the vector space $\mathcal{A}$ has the basis \eqref{basis for A}.
The result follows.
\end{proof}

\noindent
Next, we discuss the $S_4$-action on $X^\sigma_{ij}$.

\begin{lemma}\label{lem:S_4 actionX_ij}
For distinct $i,j \in \mathbb{I}$, the group $S_4$ acts on $X^\sigma_{ij}$ as follows.
For each $\beta \in S_4$, $\beta$ sends
\begin{equation*}
	X^\sigma_{ij} \quad \longmapsto \quad X^\sigma_{\beta(i)\beta(j)}.
\end{equation*}
\end{lemma}
\begin{proof}
We have
\begin{align*}
	&&\beta(X^\sigma_{ij})
	& = \beta(x^\sigma_{ij}\mathcal{A}) && ({\rm by}\ \eqref{eq:x_ijA}) &&\\
	&&& = \beta(x^\sigma_{ij})\mathcal{A} && ({\rm by \ Remark}\ \ref{rmk:AmodL}) &&\\
	&&& = x^\sigma_{\beta(i)\beta(j)}\mathcal{A} && ({\rm by \ Lemma} \  \ref{lem:EldThm1.4(2)})&&\\
	&&& = X^\sigma_{\beta(i)\beta(j)}.
\end{align*}
The result follows.
\end{proof}

Pick mutually distinct $i,j,k \in \mathbb{I}$.
In \cite[Corollary 4.9]{2024LeeJA}, we decomposed the vector space $\tet$ into a direct sum:
\begin{equation}\label{eq:tet=ds}
	\tet = X_{ij} + X_{jk} + X_{ki}.
\end{equation}
We refer to \eqref{eq:tet=ds} as a triangular-shaped decomposition. 
In the following result, we present a variation of the triangular-shaped decomposition, called a path-shaped decomposition.
\begin{proposition}
For mutually distinct $h, i, j, k \in \mathbb{I}$, the vector space $\tet$ satisfies
\begin{equation*}
	\tet = X_{kh} + X_{hi} + X_{ij} \qquad \qquad {\rm(direct \ sum)}.
\end{equation*}
\end{proposition}
\begin{proof}
By the $S_4$-symmetry on $\tet$ and the Lie algebra isomorphism $\sigma:\tet \to \L$ in Lemma \ref{def: sigma}, it suffices to show that the vector space $\L$ satisfies 
\begin{equation}\label{eq:d.s. L(sl2)+}
	\L = X^\sigma_{03} + X^\sigma_{31} + X^\sigma_{12} \qquad {\rm (direct \ sum)}.
\end{equation}
Let $\Omega$ denote the right-hand side of the equation \eqref{eq:d.s. L(sl2)+}.
By Lemma \ref{lem:basis X_ij}, $\Omega$ is spanned by the elements:
\begin{equation*}
	x^\sigma_{ij}, \qquad 
	x^\sigma_{ij}t^n, \qquad 
	x^\sigma_{ij}(t')^n, \qquad 
	x^\sigma_{ij}(t'')^n, \qquad \qquad n \in \mathbb{N}^+,
\end{equation*}
where $(i,j) \in \{(0,3), (3,1), (1,2)\}$.
This implies that the sum $\Omega$ is direct.
We now show that $\L=\Omega$.
Recall the basis \eqref{eq:basis L(sl2)+} for $\L$.
Since $x^\sigma_{31}=z\otimes 1$ and $x^\sigma_{12} = x\otimes 1$, the elements
\begin{equation*}
	u \otimes 1, \qquad 
	u \otimes t^n, \qquad 
	u \otimes (t')^n, \qquad 
	u \otimes (t'')^n, \qquad \qquad u \in \{x,z\}, \quad n \in \mathbb{N}^+,
\end{equation*}
are contained in both $\L$ and $\Omega$.
To complete the proof, we will show the following (i), (ii):
\begin{itemize}
	\item[(i)] $\L$ contains 
	\begin{equation*}
	x^\sigma_{03}, \qquad x^\sigma_{03}t^n, \qquad x^\sigma_{03}(t')^n, \qquad x^\sigma_{03}(t'')^n, \qquad \qquad n \in \mathbb{N}^+.
	\end{equation*}
	\item[(ii)] $\Omega$ contains 
	\begin{equation*}
	y \otimes 1, \qquad y \otimes t^n, \qquad y \otimes (t')^n, \qquad y \otimes (t'')^n, \qquad \qquad n \in \mathbb{N}^+.
	\end{equation*}
\end{itemize}
\noindent
\textit{Proof of} (i): Recall $x^\sigma_{03} = y \otimes 1 + z \otimes (t-1)$.
Obviously $x^\sigma_{03} \in \L$ since the elements \eqref{eq:basis L(sl2)+} are a basis for $\L$.   
Consider the elements $\{x^\sigma_{03}t^n\}_{n \in \mathbb{N}^+}$. 
For each $n \in \mathbb{N}^+$, we have
\begin{equation}\label{pf(i):eq (1)}
	x^\sigma_{03}t^n = (y \otimes t + z \otimes (t-1))t^n = y \otimes t^{n+1} + z \otimes t^{n+1} - z \otimes t^n \in \L.
\end{equation}
Consider the elements $\{x^\sigma_{03}(t')^n\}_{n \in \mathbb{N}^+}$. 
Observe that for $n \in \mathbb{N}^+$,
\begin{equation}\label{pf eq: t(t')n}
	t(t')^n = t - (1 + t' + (t')^2 + \cdots + (t')^{n-1}).
\end{equation}
Using \eqref{pf eq: t(t')n}, we have
\begin{equation}\label{pf(i):eq (2)}
\begin{split}
	x^\sigma_{03}(t')^n = (y \otimes t + z \otimes (t-1))(t')^n = & \ y \otimes t - y \otimes (1 + t' + \cdots + (t')^{n-1}) \\
	& \quad + z \otimes t - z \otimes (1 + t' + \cdots + (t')^{n}) \in \L.
\end{split}
\end{equation}
Consider the elements $\{x^\sigma_{03}(t'')^n\}_{n \in \mathbb{N}^+}$. 
Observe that for $n \in \mathbb{N}^+$,
\begin{equation}\label{pf eq: t(t'')n}
	t(t'')^n = (t'')^n - (t'')^{n-1}.
\end{equation}
Using \eqref{pf eq: t(t'')n}, we have
\begin{equation}\label{pf(i):eq (3)}
	x^\sigma_{03}(t'')^n = (y \otimes t + z \otimes (t-1))(t'')^n =  y \otimes (t'')^n - y \otimes (t'')^{n-1} - z \otimes (t'')^{n-1} \in \L.
\end{equation}
By \eqref{pf(i):eq (1)}, \eqref{pf(i):eq (2)}, \eqref{pf(i):eq (3)}, the result (i) follows.
Consequently, it follows that $\Omega \subseteq \L$.

\smallskip
\noindent
\textit{Proof of} (ii): Consider the elements $\{ y \otimes t^n\}_{n \in \mathbb{N}^+}$.
Since $x^\sigma_{03}t^n, z\otimes t^n \in \Omega$ for $n \in \mathbb{N}$, from \eqref{pf(i):eq (1)} we have
\begin{equation}\label{pf(ii):eq (1)}
	y \otimes t^{n+1} =  x^\sigma_{03}t^n - z \otimes (t-1)t^{n} \in \Omega \qquad  
	(n \in \mathbb{N}).
\end{equation}
Consider the elements $y \otimes 1$ and $\{y \otimes (t')^n\}_{n \in \mathbb{N}^+}$.
Since $x^\sigma_{03}(t')^n, y \otimes t, z \otimes t, z\otimes (t')^n \in \Omega$ for $n \in \mathbb{N}$, from \eqref{pf(i):eq (2)} we have
\begin{equation}\label{pf(ii):eq (2)}
	 y \otimes (1 + t' + \cdots + (t')^{n-1}) 	= - x^\sigma_{03}(t')^n + y \otimes t + z \otimes t - z \otimes (1 + t' + \cdots + (t')^n) \in \Omega \qquad (n\in \mathbb{N}^+).
\end{equation}
Consider the elements $\{y \otimes (t'')^n\}_{n \in \mathbb{N}^+}$.
Since $x^\sigma_{03}(t'')^n, z\otimes (t'')^n \in \Omega$ for $n \in \mathbb{N}^+$, from \eqref{pf(i):eq (3)} we have
\begin{equation}\label{pf(ii):eq (3)}
	y \otimes (t'')^n 	- y \otimes (t'')^{n-1} = x^\sigma_{03}(t'')^n + z \otimes (t'')^{n-1} \in \Omega \qquad  (n \in \mathbb{N}^+).
\end{equation}
Since $y \otimes 1 \in \Omega$, it follows inductively from \eqref{pf(ii):eq (3)} that $y \otimes (t'')^n \in \Omega$ for $n \in \mathbb{N}^+$.
By this comment together with \eqref{pf(ii):eq (1)}--\eqref{pf(ii):eq (3)}, the result (ii) follows.
Consequently, it follows that $\L \subseteq \Omega$.
The proof is complete.
\end{proof}

\section{Direct sum decompositions of $O$}\label{sec:DSO}
In the previous section, we showed that for mutually distinct $h, i, j, k \in \mathbb{I}$, the vector space $\tet$ decomposes as a direct sum of three subspaces $X_{ij}$, $X_{jk}$, and $X_{kh}$. 
By the $S_4$-symmetry on $\tet$, without loss of generality we consider the following path-shaped decomposition: 
\begin{equation*} 
	\tet = X_{03} + X_{31} + X_{12}. 
\end{equation*}
In this section, we will use this decomposition to analyze the structure of the Onsager Lie algebra $O$ and present four path-shaped decompositions of $O$. 
To this end, we carry out our computations in $\L$. 
Recall the Onsager subalgebra $O$ of $\L$ from \eqref{decomp O sigma}.


\begin{lemma}\label{lem:X_ij cap O}
The following {\rm(i)}--{\rm(iii)} hold.
\begin{enumerate}\setlength\itemsep{0em}
	\item[\rm(i)] $X^\sigma_{03} \cap O$ has a basis $\{(y \otimes t + z\otimes (t-1))t^n \}_{n \in \mathbb{N}}$.
	\item[\rm(ii)] $X^\sigma_{31} \cap O$ has a basis $\{z \otimes (t-1)t^n \}_{n \in \mathbb{N}}$.
	\item[\rm(iii)] $X^\sigma_{12} \cap O$ has a basis $\{x \otimes t^n \}_{n \in \mathbb{N}}$.
\end{enumerate}
\end{lemma}
\begin{proof}
(i): Recall $x^\sigma_{03}=y\otimes t + z\otimes (t-1)$.
By Lemma \ref{lem:basis X_ij}, $X^\sigma_{03}$ has a basis
\begin{equation}\label{eq:basisXsigma_03}
	x^\sigma_{03}, \qquad 
	x^\sigma_{03}t^n, \qquad 
	x^\sigma_{03}(t')^n, \qquad 
	x^\sigma_{03}(t'')^n, \qquad \qquad n \in \mathbb{N}^+.
\end{equation}
Since $O$ has the basis \eqref{eq:basis for Osigma}, elements of the form $x^\sigma_{03} (t')^n$ or $x^\sigma_{03} (t'')^n$ in \eqref{eq:basisXsigma_03} cannot belong to $O$.
Thus, it follows that $\Span\{x^\sigma_{03} t^n  \mid n \in \mathbb{N}\} \subseteq X^\sigma_{03} \cap O$.
For the reverse inclusion, let $u \in X^\sigma_{03}\cap O$.
Write $u = (y\otimes t + z\otimes (t-1))a$ for some $a\in \mathcal{A}$. 
Since $u \in O$ and by \eqref{decomp O sigma}, we have $ta \in t\mathbb{F}[t]$ and $(t-1)a \in (t-1)\mathbb{F}[t]$.
This implies that $a$ must take the form of a polynomial in $t$. 
Hence, it follows $u\in \Span\{x^\sigma_{03} t^n \mid n \in \mathbb{N}\}$.
Therefore, we have $\Span\{x^\sigma_{03} t^n \mid n \in \mathbb{N}\} = X^\sigma_{03} \cap O$. 
Since the set $\{x^\sigma_{03} t^n \}_{n \in \mathbb{N}}$ is linearly independent, the result follows.

\noindent
(ii), (iii): Similar to (i).
\end{proof}

\begin{lemma}\label{lem:ds O [0312]}
We have
\begin{equation}\label{eq(1):dsO[0312]}
	O = (X^\sigma_{03} \cap O) + (X^\sigma_{31} \cap O) + (X^\sigma_{12}\cap O) \qquad \text{\rm (direct sum)}.
\end{equation}
\end{lemma}
\begin{proof}
Let $\Delta$ denote the right-hand side of \eqref{eq(1):dsO[0312]}.
By Lemma \ref{lem:X_ij cap O}, the sum $\Delta$ is direct.
We now show that $O=\Delta$.
From Lemma \ref{lem:X_ij cap O}, the elements
\begin{equation}\label{pf:eq:basis O}
	(y \otimes t + z\otimes (t-1))t^n, \qquad 
	z \otimes (t-1)t^n, \qquad 
	x \otimes t^n, \qquad \qquad n \in \mathbb{N},
\end{equation}
form a basis for $\Delta$.
Consider the basis \eqref{eq:basis for Osigma} for $O$.
By comparing the two bases \eqref{eq:basis for Osigma} and \eqref{pf:eq:basis O}, we find that $O=\Delta$.
The result follows.
\end{proof}

The following result is a variation of Lemma \ref{lem:ds O [0312]}.

\begin{lemma}\label{lem:ds O variations}
The Onsager Lie algebra $O$ satisfies
\begin{align}
	O 
	& =  (X^\sigma_{30} \cap O) + (X^\sigma_{02} \cap O) + (X^\sigma_{21}\cap O) \label{eq(1):dsO[3021]}\\
	& =  (X^\sigma_{03} \cap O) + (X^\sigma_{32} \cap O) + (X^\sigma_{21}\cap O) \label{eq(1):dsO[0321]}\\
	& =  (X^\sigma_{30} \cap O) + (X^\sigma_{01} \cap O) + (X^\sigma_{12}\cap O).\label{eq(1):dsO[3012]}
\end{align}
Each of the sums \eqref{eq(1):dsO[3021]}--\eqref{eq(1):dsO[3012]} is direct.
\end{lemma}
\begin{proof}
Recall the automorphisms $\rho, \tau$ of $\L$ from Lemma \ref{lem:EldThm1.4} and the group $G=\<\rho, \tau\>$ from Definition \ref{def:Z2xZ2-group}.
Note that $\rho=(12)(30)$, $\tau=(12)$, $\rho\tau=(30)$ and that $O$ is invariant under the $G$-action.
Using Lemma \ref{lem:S_4 actionX_ij}, apply $\rho$, $\tau$, $\rho\tau$ to \eqref{eq(1):dsO[0312]} to obtain \eqref{eq(1):dsO[3021]}, \eqref{eq(1):dsO[0321]}, \eqref{eq(1):dsO[3012]}, respectively.
The result follows.
\end{proof}

\begin{corollary}
In $\tet$, we have
\begin{align}
	O 
	& = (X_{03} \cap O) + (X_{31} \cap O) + (X_{12}\cap O) \label{tet:dsO[3021]}\\
	& =  (X_{30} \cap O) + (X_{02} \cap O) + (X_{21}\cap O) \label{tet:dsO[3021]}\\
	& =  (X_{03} \cap O) + (X_{32} \cap O) + (X_{21}\cap O) \label{tet:dsO[0321]}\\
	& =  (X_{30} \cap O) + (X_{01} \cap O) + (X_{12}\cap O).\label{tet:dsO[3012]}
\end{align}
Each of the sums \eqref{tet:dsO[3021]}--\eqref{tet:dsO[3012]} is direct.
\end{corollary}
\begin{proof}
By Lemmas \ref{lem:ds O [0312]}, \ref{lem:ds O variations}.
\end{proof}

\section{A basis $[0312]$ for $O$}\label{sec:0312}

In this section, we find an attractive basis for $O$ consisting of $x_{03}$-like or $x_{31}$-like or $x_{12}$-like elements. 
We will denote this basis by $[0312]$. 
We then describe the action of the Lie bracket on this basis.
First, recall the direct sum decomposition of $O$ from \eqref{eq(1):dsO[0312]}.

\begin{lemma}\label{lem:0312 O=d.s}
We have
\begin{equation*}
	O = x\otimes \mathbb{F}[t] + (y \otimes t + z\otimes (t-1))\mathbb{F}[t] + z\otimes (t-1)\mathbb{F}[t] \qquad ({\rm direct \ sum}).
\end{equation*}
Moreover, 
\begin{equation}\label{eq(2):0312 O=d.s}
	X^\sigma_{12} \cap O = x \otimes \mathbb{F}[t], \quad	
	X^\sigma_{31} \cap O = z \otimes (t-1)\mathbb{F}[t], \quad 
	X^\sigma_{03} \cap O = (y \otimes t + z\otimes (t-1))\mathbb{F}[t].
\end{equation}
\end{lemma}
\begin{proof}
By Lemmas \ref{lem:X_ij cap O} and \ref{lem:ds O [0312]}. 
\end{proof}

Define the elements $A^{\uparrow\uparrow}_i, B^{\uparrow\uparrow}_i, \psi^{\uparrow\uparrow}_i \in \L$ as follows.
\begin{align}
	&& A^{\uparrow\uparrow}_i & :=  x \otimes (t-1)^i && i\geq 0, && \label{def:A(i)}\\
	&& B^{\uparrow\uparrow}_i & :=  (y \otimes t + z \otimes (t-1)) (t-1)^i && i\geq 0, && \label{def:B(i)}\\
	&& \psi^{\uparrow\uparrow}_i & :=  z \otimes (t-1)^i && i\geq 0. &&\label{def:psi(i)}
\end{align}
Observe that 
\begin{equation*}
	A^{\uparrow\uparrow}_0 = x^\sigma_{12}, \qquad 
	B^{\uparrow\uparrow}_0 = x^\sigma_{03}, \qquad
	\psi^{\uparrow\uparrow}_0 = x^\sigma_{31}.
\end{equation*}
Recall the standard generators $A$, $B$ of $O$.
Note that $A^{\uparrow\uparrow}_0=A$ and $B^{\uparrow\uparrow}_0=B$.
Note also that $\psi^{\uparrow\uparrow}_0$ is not in $O$.

\begin{lemma}\label{lem:0312 basis}
Referring to \eqref{def:A(i)}--\eqref{def:psi(i)}, the following {\rm(i)}--{\rm(iii)} hold.
\begin{itemize}
	\item[\rm(i)] The elements $\{A^{\uparrow\uparrow}_i\}_{i \in \mathbb{N}}$ form a basis for $X^\sigma_{12} \cap O$.
	\item[\rm(ii)] The elements $\{B^{\uparrow\uparrow}_i\}_{i \in \mathbb{N}}$ form a basis for $X^\sigma_{03} \cap O$.
	\item[\rm(iii)] The elements $\{\psi^{\uparrow\uparrow}_{i+1}\}_{i \in \mathbb{N}}$ form a basis for $X^\sigma_{31} \cap O$.
\end{itemize}
\end{lemma}
\begin{proof}
(i): From the first equation in \eqref{eq(2):0312 O=d.s}, we show that the elements $\{x\otimes(t-1)^i\}_{i \in \mathbb{N}}$ form a basis for $x\otimes \mathbb{F}[t]$.
Since $\{(t-1)^i\}_{i \in \mathbb{N}}$ is a basis for $\mathbb{F}[t]$, the result follows readily.\\
(ii): Similar to (i).\\
(iii): Similar to (i). Note that $\psi^{\uparrow\uparrow}_0 = z\otimes 1 \notin z\otimes(t-1)\mathbb{F}[t]$.
\end{proof}


\begin{theorem}\label{thm:0312 basis}
Referring to \eqref{def:A(i)}--\eqref{def:psi(i)}, the Onsager Lie algebra $O$ has a basis
\begin{equation}\label{basis:Ai Bi phi}
	A^{\uparrow\uparrow}_i, \qquad 
	B^{\uparrow\uparrow}_i, \qquad 
	\psi^{\uparrow\uparrow}_{i+1}, \qquad \qquad i\in \mathbb{N},
\end{equation}
where $A^{\uparrow\uparrow}_0 = A$, $B^{\uparrow\uparrow}_0 = B$.
The Lie bracket acts on this basis as follows. 
For $i,j \in \mathbb{N}$,
\begin{align}
	[A^{\uparrow\uparrow}_i, A^{\uparrow\uparrow}_j] & = 0, \label{thm0312:eq[A,A]}\\
	[B^{\uparrow\uparrow}_i, B^{\uparrow\uparrow}_j] & = 0, \label{thm0312:eq[B,B]}\\
	[\psi^{\uparrow\uparrow}_i, \psi^{\uparrow\uparrow}_j] & = 0, \label{thm0312:eq[psi,psi]}\\
	[\psi^{\uparrow\uparrow}_i, A^{\uparrow\uparrow}_j] & = 2\psi^{\uparrow\uparrow}_{i+j} + 2A^{\uparrow\uparrow}_{i+j}, \label{thm0312:eq[psi,A]}\\
	[B^{\uparrow\uparrow}_i, \psi^{\uparrow\uparrow}_j] & = 2B^{\uparrow\uparrow}_{i+j} + 2\psi^{\uparrow\uparrow}_{i+j}, \label{thm0312:eq[B,psi]}\\
	[A^{\uparrow\uparrow}_i, B^{\uparrow\uparrow}_j] & = 2A^{\uparrow\uparrow}_{i+j} + 2B^{\uparrow\uparrow}_{i+j} - 4\psi^{\uparrow\uparrow}_{i+j+1}. \label{thm0312:eq[A,B]}
\end{align}
\end{theorem}
\begin{proof}
By Lemmas \ref{lem:0312 O=d.s} and \ref{lem:0312 basis}, the elements in \eqref{basis:Ai Bi phi} form a basis for $O$.
We now prove \eqref{thm0312:eq[A,A]}--\eqref{thm0312:eq[A,B]}.
First, equations \eqref{thm0312:eq[A,A]}--\eqref{thm0312:eq[psi,psi]} follow from the fact that $[u,u]=0$ for all $u \in \mathfrak{sl}_2$. 
Next, we show \eqref{thm0312:eq[psi,A]}.
Using \eqref{eq:[x,y]}, \eqref{def:eq[u,v]x(ab)}, \eqref{def:A(i)} and \eqref{def:psi(i)}, we obtain
\begin{equation*}
	[\psi^{\uparrow\uparrow}_i, A^{\uparrow\uparrow}_j]
	= [z,x] \otimes (t-1)^{i+j}
	= 2z \otimes (t-1)^{i+j} + 2x \otimes (t-1)^{i+j}
	= 2\psi^{\uparrow\uparrow}_{i+j} + 2 A^{\uparrow\uparrow}_{i+j},
\end{equation*}
as desired.
Equations \eqref{thm0312:eq[B,psi]}, \eqref{thm0312:eq[A,B]} are similiarly obtained.
\end{proof}

\begin{notation}
We denote by $[0312]$ the basis \eqref{basis:Ai Bi phi} for $O$. 
\end{notation}

\begin{corollary}\label{cor:[0312] rec}
The basis $[0312]$ for $O$ is recursively obtained in the following order
$$
	A^{\uparrow\uparrow}_0, \quad
	B^{\uparrow\uparrow}_0, \quad
	\psi^{\uparrow\uparrow}_1, \quad
	A^{\uparrow\uparrow}_1, \quad
	B^{\uparrow\uparrow}_1, \quad
	\psi^{\uparrow\uparrow}_2, \quad
	A^{\uparrow\uparrow}_2, \quad
	B^{\uparrow\uparrow}_2, \quad
	\psi^{\uparrow\uparrow}_3,  \quad \ldots
$$
using $A^{\uparrow\uparrow}_0= A$, $B^{\uparrow\uparrow}_0=B$ and the following equations for $i \geq 1$:
\begin{align}
	\psi^{\uparrow\uparrow}_{i} & = \frac{A^{\uparrow\uparrow}_{i-1}}{2} + \frac{B^{\uparrow\uparrow}_{i-1}}{2} - \frac{[A^{\uparrow\uparrow}_{i-1}, B]}{4}, \label{cor:0312 psi}\\
	A^{\uparrow\uparrow}_{i} & = \frac{[\psi^{\uparrow\uparrow}_{i}, A]}{2} - \psi^{\uparrow\uparrow}_{i}, \label{cor:0312 A}\\
	B^{\uparrow\uparrow}_{i} & = \frac{[B, \psi^{\uparrow\uparrow}_{i}]}{2} - \psi^{\uparrow\uparrow}_{i}. \label{cor:0312 B}
\end{align}
\end{corollary}
\begin{proof}
Set $j=0$ in \eqref{thm0312:eq[A,B]} and solve for $\psi^{\uparrow\uparrow}_{i+1}$ to get \eqref{cor:0312 psi}.
Similarly, set $j=0$ in \eqref{thm0312:eq[psi,A]} and solve for $A^{\uparrow\uparrow}_{i}$ to get \eqref{cor:0312 A}, and in \eqref{thm0312:eq[B,psi]} to solve for $B^{\uparrow\uparrow}_{i}$ to get \eqref{cor:0312 B}. 
The proof is complete.
\end{proof}

\begin{example}
Recall the basis elements of $[0312]$ in \eqref{basis:Ai Bi phi}.
We express $A^{\uparrow\uparrow}_{i}$, $B^{\uparrow\uparrow}_{i}$, $\psi^{\uparrow\uparrow}_{i+1}$ for $0\leq i \leq 2$.
Observe that
\begin{equation*}
	A^{\uparrow\uparrow}_0 =  A, \qquad \qquad B^{\uparrow\uparrow}_0 = B.
\end{equation*}
By Corollary \ref{cor:[0312] rec}, 
\begin{align*}
	\psi^{\uparrow\uparrow}_{1} & = \frac{A}{2} + \frac{B}{2} - \frac{1}{4}[A,B], \\
	A^{\uparrow\uparrow}_{1} & = -\frac{A}{2} -\frac{B}{2} - \frac{1}{8}[A,[B,A]], \\
	B^{\uparrow\uparrow}_{1} & = -\frac{A}{2} -\frac{B}{2} - \frac{1}{8}[B,[A,B]], \\
	\psi^{\uparrow\uparrow}_{2} & = -\frac{A}{2} - \frac{B}{2} - \frac{1}{8}[B,A] - \frac{1}{16} [A, [B, A]] - \frac{1}{16}[B,[A,B]] - \frac{1}{32} [B, [A, [B, A]]], \\
	A^{\uparrow\uparrow}_{2} & = \frac{A}{2} + \frac{B}{2} + \frac{1}{8}[A, [B, A]] + \frac{1}{16} [B, [A, B]] + \frac{1}{64} [A, [B, [A, [B, A]]]], \\
	B^{\uparrow\uparrow}_{2} & = \frac{A}{2} + \frac{B}{2} + \frac{1}{8}[B, [A, B]] + \frac{1}{16} [A, [B, A]] + \frac{1}{64} [B, [A, [B, [A, B]]]],\\
	\psi^{\uparrow\uparrow}_{3} & = \frac{A}{2} + \frac{B}{2} + \frac{1}{16}[B,A] + \frac{3}{32}[A,[B,A]] + \frac{3}{32}[B,[A,B]] + \frac{1}{32}[B, [A,[B,A]]] \\
	& \qquad + \frac{1}{128} [A,[B,[A,[B,A]]]] + \frac{1}{128} [B,[A,[B,[A,B]]]] + \frac{1}{256}[B,[A,[B,[A,[B,A]]]]].
\end{align*}
\end{example}

\section{A basis $[3021]$ for $O$}\label{sec:3021}
In this section, we find an attractive basis for $O$ consisting of $x_{30}$-like or $x_{02}$-like or $x_{21}$-like elements.
We will denote this basis by $[3021]$. 
We then describe the action of Lie bracket on this basis.
First, recall the direct sum decomposition of $O$ from \eqref{eq(1):dsO[3021]}. 

\begin{lemma}\label{lem:3021,O=d.s}
We have
\begin{equation}\label{eq:3021 O=d.s}
	O = x\otimes \mathbb{F}[t] + (x\otimes1 + y \otimes t)\mathbb{F}[t] + (y \otimes t + z\otimes (t-1))\mathbb{F}[t]  \qquad ({\rm direct \ sum}).
\end{equation}
Moreover, 
\begin{equation}\label{eq(2):3021 O=d.s}
	X^\sigma_{21} \cap O = x \otimes \mathbb{F}[t], \quad	
	X^\sigma_{02} \cap O = (x\otimes1 + y \otimes t)\mathbb{F}[t], \quad 
	X^\sigma_{30} \cap O = (y \otimes t + z\otimes (t-1))\mathbb{F}[t].
\end{equation}
\end{lemma}
\begin{proof}
First, we show \eqref{eq(2):3021 O=d.s}.
Recall the group $G$ from Definition \ref{def:Z2xZ2-group} and the permutation $\rho=(12)(03)\in G$.
By Lemma \ref{lem:S_4 actionX_ij}, $\rho$ sends $X^\sigma_{12} \cap O \to X^\sigma_{21} \cap O$.
Since $X^\sigma_{12}=X^\sigma_{21}$ and by the first equation in \eqref{eq(2):0312 O=d.s}, we obtain the first equation in \eqref{eq(2):3021 O=d.s}.
Similarly, we obtain the third equation in \eqref{eq(2):3021 O=d.s}.
For the second equation in \eqref{eq(2):3021 O=d.s}, apply $\rho$ to both sides of the second equation in \eqref{eq(2):0312 O=d.s}. 
Then we observe that $\rho(X^\sigma_{31} \cap O) = X^\sigma_{02}\cap O$ and
\begin{align*}
	&&\rho\left((z\otimes(t-1)\mathbb{F}[t]\right) 
	& = \rho(z\otimes1)(t-1)\mathbb{F}[t] && (\text{\rm by \eqref{eq:action_rho L} and Remark \ref{rmk:AmodL}})&&\\
	&&& = (-x\otimes1-y\otimes t)\mathbb{F}[t] && (\text{\rm by \eqref{eq:action_rho(z)}})&&\\
	&&& = (x\otimes1+y\otimes t)\mathbb{F}[t].&&
\end{align*}
Therefore, we have shown \eqref{eq(2):3021 O=d.s}.
Line \eqref{eq:3021 O=d.s} follows from \eqref{eq(1):dsO[3021]} and \eqref{eq(2):3021 O=d.s}.
\end{proof}
Define the elements $A^{\downarrow\downarrow}_i, B^{\downarrow\downarrow}_i, \psi^{\downarrow\downarrow}_i \in \L$ as follows.
\begin{align}
	&& A^{\downarrow\downarrow}_{i} & :=  - x \otimes (t-1)^i && i\geq 0, && \label{[3021]: A}\\
	&& B^{\downarrow\downarrow}_{i} & :=  (- y \otimes t - z \otimes (t-1)) (t-1)^i && i\geq 0, && \label{[3021]: B}\\
	&& \psi^{\downarrow\downarrow}_{i} & :=  (-x\otimes 1 - y \otimes t) (t-1)^{i-1} && i\geq 0. && \label{[3021]: psi}
\end{align}
Observe that 
\begin{equation*}
	A^{\downarrow\downarrow}_0 = x^\sigma_{21}, \qquad 
	B^{\downarrow\downarrow}_{0} = x^\sigma_{30}, \qquad
	\psi^{\downarrow\downarrow}_{0}=x^\sigma_{02}.
\end{equation*}
Recall the standard generators $A$, $B$ of $O$.
Note that $A^{\downarrow\downarrow}_{0}=-A$ and $B^{\downarrow\downarrow}_{0}=-B$.
Note also that $\psi^{\downarrow\downarrow}_{0}$ is not in $O$.

\begin{lemma}\label{lem:3021 basis}
Referring to \eqref{[3021]: A}--\eqref{[3021]: psi}, the following {\rm(i)}--{\rm(iii)} hold.
\begin{itemize}
	\item[\rm(i)] The elements $\{A^{\downarrow\downarrow}_{i}\}_{i \in \mathbb{N}}$ form a basis for $X^\sigma_{21} \cap O$.
	\item[\rm(ii)] The elements $\{B^{\downarrow\downarrow}_{i}\}_{i \in \mathbb{N}}$ form a basis for $X^\sigma_{30} \cap O$.
	\item[\rm(iii)] The elements $\{\psi^{\downarrow\downarrow}_{i+1}\}_{i \in \mathbb{N}}$ form a basis for $X^\sigma_{02} \cap O$.
\end{itemize}
\end{lemma}
\begin{proof}
Similar to Lemma \ref{lem:0312 basis}.
\end{proof}

\begin{theorem}\label{thm:3021 basis}
Referring to \eqref{[3021]: A}--\eqref{[3021]: psi}, the Onsager Lie algebra $O$ has a basis
\begin{equation}\label{basis[3021]}
	A^{\downarrow\downarrow}_{i}, \qquad 
	B^{\downarrow\downarrow}_{i}, \qquad 
	\psi^{\downarrow\downarrow}_{i+1}, \qquad \qquad i\in \mathbb{N},
\end{equation}
where $A^{\downarrow\downarrow}_0=-A$, $B^{\downarrow\downarrow}_0=-B$.
The Lie bracket acts on this basis as follows. 
For $i,j \in \mathbb{N}$,
\begin{align}
	[A^{\downarrow\downarrow}_i, A^{\downarrow\downarrow}_j] & = 0, \label{thm3021:eq[A,A]}\\
	[B^{\downarrow\downarrow}_i, B^{\downarrow\downarrow}_j] & = 0, \\
	[\psi^{\downarrow\downarrow}_i, \psi^{\downarrow\downarrow}_j] & = 0, \\
	[\psi^{\downarrow\downarrow}_i, A^{\downarrow\downarrow}_j] & = 2\psi^{\downarrow\downarrow}_{i+j} + 2A^{\downarrow\downarrow}_{i+j}, \\
	[B^{\downarrow\downarrow}_{i}, \psi^{\downarrow\downarrow}_{j}] & = 2B^{\downarrow\downarrow}_{i+j} + 2\psi^{\downarrow\downarrow}_{i+j}, \\
	[A^{\downarrow\downarrow}_{i}, B^{\downarrow\downarrow}_{j}] & = 2A^{\downarrow\downarrow}_{i+j} + 2B^{\downarrow\downarrow}_{i+j} - 4\psi^{\downarrow\downarrow}_{i+j+1}. \label{thm3021:eq[A,B]}
\end{align}
\end{theorem}
\begin{proof}
By Lemmas \ref{lem:3021,O=d.s} and \ref{lem:3021 basis}, the elements in \eqref{basis[3021]} form a basis for $O$.
Equations \eqref{thm3021:eq[A,A]}--\eqref{thm3021:eq[A,B]} are obtained in a manner similar to that used in Theorem \ref{thm:0312 basis}.
\end{proof}

\begin{notation}
We denote by $[3021]$ the basis \eqref{basis[3021]} for $O$. 
\end{notation}

\begin{corollary}\label{cor:[3021] rec}
The basis $[3021]$ for $O$ is recursively obtained in the following order
$$
	A^{\downarrow\downarrow}_0, \quad
	B^{\downarrow\downarrow}_0, \quad
	\psi^{\downarrow\downarrow}_1, \quad
	A^{\downarrow\downarrow}_1, \quad
	B^{\downarrow\downarrow}_1, \quad
	\psi^{\downarrow\downarrow}_2, \quad
	A^{\downarrow\downarrow}_2, \quad
	B^{\downarrow\downarrow}_2, \quad
	\psi^{\downarrow\downarrow}_3,  \quad \ldots 
$$
using $A^{\downarrow\downarrow}_0= -A$, $B^{\downarrow\downarrow}_0=-B$ and the following equations for $i \geq 1$:
\begin{align*}
	\psi^{\downarrow\downarrow}_{i} & = \frac{A^{\downarrow\downarrow}_{i-1}}{2} + \frac{B^{\downarrow\downarrow}_{i-1}}{2} + \frac{[A^{\downarrow\downarrow}_{i-1}, B]}{4}, \\
	A^{\downarrow\downarrow}_{i} & = - \frac{[\psi^{\downarrow\downarrow}_{i}, A]}{2} - \psi^{\downarrow\downarrow}_{i}, \\ 
	B^{\downarrow\downarrow}_{i} & = - \frac{[B, \psi^{\downarrow\downarrow}_{i}]}{2} - \psi^{\downarrow\downarrow}_{i}.  
\end{align*}
\end{corollary}
\begin{proof}
Similar to Corollary \ref{cor:[0312] rec}.
\end{proof}

\begin{example}
Recall the basis elements of $[3021]$ in \eqref{basis[3021]}.
We express $A^{\downarrow\downarrow}_{i}$, $B^{\downarrow\downarrow}_{i}$, $\psi^{\downarrow\downarrow}_{i+1}$ for $0\leq i \leq 2$.
Observe that
\begin{equation*}
	A^{\downarrow\downarrow}_0 =  -A, \qquad \qquad B^{\downarrow\downarrow}_0 = -B.
\end{equation*}
By Corollary \ref{cor:[3021] rec}, 
\begin{align*}
	\psi^{\downarrow\downarrow}_{1} & = - \frac{A}{2} - \frac{B}{2} - \frac{1}{4}[A,B], \\
	A^{\downarrow\downarrow}_{1} & = \frac{A}{2} + \frac{B}{2} + \frac{1}{8}[A,[B,A]], \\
	B^{\downarrow\downarrow}_{1} & = \frac{A}{2} + \frac{B}{2} + \frac{1}{8}[B,[A,B]], \\
	\psi^{\downarrow\downarrow}_{2} & = \frac{A}{2} + \frac{B}{2} - \frac{1}{8}[B,A] + \frac{1}{16} [A, [B, A]] + \frac{1}{16}[B,[A,B]] - \frac{1}{32} [B, [A, [B, A]]], \\
	A^{\downarrow\downarrow}_{2} & = - \frac{A}{2} - \frac{B}{2} - \frac{1}{8}[A, [B, A]] - \frac{1}{16} [B, [A, B]] - \frac{1}{64} [A, [B, [A, [B, A]]]], \\
	B^{\downarrow\downarrow}_{2} & = - \frac{A}{2} - \frac{B}{2} - \frac{1}{8}[B, [A, B]] - \frac{1}{16} [A, [B, A]] - \frac{1}{64} [B, [A, [B, [A, B]]]],\\
	\psi^{\downarrow\downarrow}_{3} & = - \frac{A}{2} - \frac{B}{2} + \frac{1}{16}[B,A] - \frac{3}{32}[A,[B,A]] - \frac{3}{32}[B,[A,B]] + \frac{1}{32}[B, [A,[B,A]]] \\
	& \qquad - \frac{1}{128} [A,[B,[A,[B,A]]]] - \frac{1}{128} [B,[A,[B,[A,B]]]] + \frac{1}{256}[B,[A,[B,[A,[B,A]]]]].
\end{align*}
\end{example}

\section{A basis $[0321]$ for $O$}\label{sec:0321}
In this section, we find an attractive basis for $O$ consisting of $x_{03}$-like or $x_{32}$-like or $x_{21}$-like elements. 
We will denote this basis by $[0321]$. 
We then describe the action of Lie bracket on this basis.
First, recall the direct sum decomposition of $O$ from \eqref{eq(1):dsO[0321]}.

\begin{lemma}\label{lem:0321,O=d.s}
We have
\begin{equation*}
	O = x\otimes \mathbb{F}[t] + y\otimes t\mathbb{F}[t] + (y \otimes t + z \otimes (t-1)) \mathbb{F}[t] \qquad ({\rm direct \ sum}).
\end{equation*}
Moreover, 
\begin{equation*}
	X^\sigma_{21} \cap O = x \otimes \mathbb{F}[t], \quad	
	X^\sigma_{32} \cap O = t\mathbb{F}[t], \quad 
	X^\sigma_{03} \cap O = (y \otimes t + z\otimes (t-1))\mathbb{F}[t].
\end{equation*}
\end{lemma}
\begin{proof}
Recall the group $G$ from Definition \ref{def:Z2xZ2-group} and the permutation $\tau=(12)\in G$.
Apply $\tau$ to each equation in \eqref{eq(2):0312 O=d.s} in the same manner as in Lemma \ref{lem:3021,O=d.s}.
\end{proof}

Define the elements $A^{\downarrow\uparrow}_i, B^{\downarrow\uparrow}_i, \psi^{\downarrow\uparrow}_i \in \L$ as follows.
\begin{align}
	&& A^{\downarrow\uparrow}_{i} & :=  - x \otimes (-t)^i && i\geq 0, && \label{[0321]: A}\\
	&& B^{\downarrow\uparrow}_i & :=  (y \otimes t + z \otimes (t-1)) (-t)^i && i\geq 0, && \label{[0321]: B}\\
	&& \psi^{\downarrow\uparrow}_{i} & :=  - y \otimes (-t)^{i} && i\geq 0. && \label{[0321]: psi}
\end{align}
Observe that 
\begin{equation*}
	A^{\downarrow\uparrow}_0 = x^\sigma_{21}, \qquad 
	B^{\downarrow\uparrow}_{0} = x^\sigma_{03}, \qquad
	\psi^{\downarrow\uparrow}_{0}=x^\sigma_{32}.
\end{equation*}
Recall the standard generators $A$, $B$ of $O$.
Note that $A^{\downarrow\uparrow}_{0}=-A$ and $B^{\downarrow\uparrow}_{0}=B$.
Note also that $\psi^{\downarrow\uparrow}_{0}$ is not in $O$.

\begin{lemma}\label{lem:0321 basis}
Referring to \eqref{[0321]: A}--\eqref{[0321]: psi}, the following {\rm(i)}--{\rm(iii)} hold.
\begin{itemize}
	\item[\rm(i)] The elements $\{A^{\downarrow\uparrow}_{i}\}_{i \in \mathbb{N}}$ form a basis for $X^\sigma_{21} \cap O$.
	\item[\rm(ii)] The elements $\{B^{\downarrow\uparrow}_{i}\}_{i \in \mathbb{N}}$ form a basis for $X^\sigma_{03} \cap O$.
	\item[\rm(iii)] The elements $\{\psi^{\downarrow\uparrow}_{i+1}\}_{i \in \mathbb{N}}$ form a basis for $X^\sigma_{32} \cap O$.
\end{itemize}
\end{lemma}
\begin{proof}
Similar to Lemma \ref{lem:0312 basis}.
\end{proof}

\begin{theorem}\label{thm:0321 basis}
Referring to \eqref{[0321]: A}--\eqref{[0321]: psi}, the Onsager Lie algebra $O$ has a basis
\begin{equation}\label{basis[0321]}
	A^{\downarrow\uparrow}_{i}, \qquad 
	B^{\downarrow\uparrow}_{i}, \qquad 
	\psi^{\downarrow\uparrow}_{i+1}, \qquad \qquad i\in \mathbb{N},
\end{equation}
where $A^{\downarrow\uparrow}_0 = - A$, $B^{\downarrow\uparrow}_0 = B$.
The Lie bracket acts on this basis as follows. 
For $i,j \in \mathbb{N}$,
\begin{align}
	[A^{\downarrow\uparrow}_{i}, A^{\downarrow\uparrow}_{j}] & = 0, \label{thm0321:eq[A,A]}\\
	[B^{\downarrow\uparrow}_{i}, B^{\downarrow\uparrow}_{j}] & = 0, \\
	[\psi^{\downarrow\uparrow}_{i}, \psi^{\downarrow\uparrow}_{j}] & = 0, \\
	[\psi^{\downarrow\uparrow}_{i}, A^{\downarrow\uparrow}_{j}] & = 2\psi^{\downarrow\uparrow}_{i+j} + 2A^{\downarrow\uparrow}_{i+j}, \\
	[B^{\downarrow\uparrow}_{i}, \psi^{\downarrow\uparrow}_{j}] & = 2B^{\downarrow\uparrow}_{i+j} + 2\psi^{\downarrow\uparrow}_{i+j}, \\
	[A^{\downarrow\uparrow}_{i}, B^{\downarrow\uparrow}_{j}] & = 2A^{\downarrow\uparrow}_{i+j} + 2B^{\downarrow\uparrow}_{i+j} - 4\psi^{\downarrow\uparrow}_{i+j+1}.\label{thm0321:eq[A,B]}
\end{align}
\end{theorem}
\begin{proof}
By Lemmas \ref{lem:0321,O=d.s} and \ref{lem:0321 basis}, the elements in \eqref{basis[0321]} form a basis for $O$.
Equations \eqref{thm0321:eq[A,A]}--\eqref{thm0321:eq[A,B]} are obtained in a manner similar to that used in Theorem \ref{thm:0312 basis}.
\end{proof}

\begin{notation}
We denote by $[0321]$ the basis \eqref{basis[0321]} for $O$.
\end{notation}

\begin{corollary}\label{cor:[0321] rec}
The basis $[0321]$ for $O$ is recursively obtained in the following order
$$
	A^{\downarrow\uparrow}_0, \quad 
	B^{\downarrow\uparrow}_0, \quad
	\psi^{\downarrow\uparrow}_1, \quad 
	A^{\downarrow\uparrow}_1, \quad 
	B^{\downarrow\uparrow}_1, \quad 
	\psi^{\downarrow\uparrow}_2, \quad 
	A^{\downarrow\uparrow}_2, \quad 
	B^{\downarrow\uparrow}_2, \quad 
	\psi^{\downarrow\uparrow}_3, \quad \ldots
$$
using $A^{\downarrow\uparrow}_0 = -A$, $B^{\downarrow\uparrow}_0 = B$ and the following equations for $i \geq 1$:
\begin{align*}
	\psi^{\downarrow\uparrow}_{i} & = \frac{A^{\downarrow\uparrow}_{i-1}}{2} + \frac{B^{\downarrow\uparrow}_{i-1}}{2} - \frac{[A^{\downarrow\uparrow}_{i-1}, B]}{4}, \\ 
	A^{\downarrow\uparrow}_{i} & = - \frac{[\psi^{\downarrow\uparrow}_{i}, A]}{2} - \psi^{\downarrow\uparrow}_{i}, \\ 
	B^{\downarrow\uparrow}_{i} & = \frac{[B, \psi^{\downarrow\uparrow}_{i}]}{2} - \psi^{\downarrow\uparrow}_{i}. 
\end{align*}
\end{corollary}
\begin{proof}
Similar to Corollary \ref{cor:[0312] rec}.
\end{proof}

\begin{example}
Recall the basis elements of $[0321]$ in \eqref{basis[0321]}.
We express $A^{\downarrow\uparrow}_{i}$, $B^{\downarrow\uparrow}_{i}$, $\psi^{\downarrow\uparrow}_{i+1}$ for $0\leq i \leq 2$.
Observe that
\begin{equation*}
	A^{\downarrow\uparrow}_0 =  -A, \qquad \qquad B^{\downarrow\uparrow}_0 = B.
\end{equation*}
By Corollary \ref{cor:[0321] rec}, 
\begin{align*}
	\psi^{\downarrow\uparrow}_{1} & = -\frac{A}{2} + \frac{B}{2} + \frac{1}{4}[A,B], \\
	A^{\downarrow\uparrow}_{1} & = \frac{A}{2} -\frac{B}{2} - \frac{1}{8}[A,[B,A]], \\
	B^{\downarrow\uparrow}_{1} & = \frac{A}{2} -\frac{B}{2} + \frac{1}{8}[B,[A,B]], \\
	\psi^{\downarrow\uparrow}_{2} & = \frac{A}{2} - \frac{B}{2} + \frac{1}{8}[B,A] - \frac{1}{16} [A, [B, A]] + \frac{1}{16}[B,[A,B]] - \frac{1}{32} [B, [A, [B, A]]], \\
	A^{\downarrow\uparrow}_{2} & = - \frac{A}{2} + \frac{B}{2} + \frac{1}{8}[A, [B, A]] - \frac{1}{16} [B, [A, B]] - \frac{1}{64} [A, [B, [A, [B, A]]]], \\
	B^{\downarrow\uparrow}_{2} & = - \frac{A}{2} + \frac{B}{2} - \frac{1}{8}[B, [A, B]] + \frac{1}{16} [A, [B, A]] + \frac{1}{64} [B, [A, [B, [A, B]]]],\\
	\psi^{\downarrow\uparrow}_{3} & = - \frac{A}{2} + \frac{B}{2} - \frac{1}{16}[B,A] + \frac{3}{32}[A,[B,A]] - \frac{3}{32}[B,[A,B]] + \frac{1}{32}[B, [A,[B,A]]] \\
	& \qquad - \frac{1}{128} [A,[B,[A,[B,A]]]] + \frac{1}{128} [B,[A,[B,[A,B]]]] - \frac{1}{256}[B,[A,[B,[A,[B,A]]]]].
\end{align*}
\end{example}

\section{A basis $[3012]$ for $O$}\label{sec:3012}
In this section, we find an attractive basis for $O$ consisting of $x_{30}$-like or $x_{01}$-like or $x_{12}$-like elements.
We will denote this basis by $[3012]$. 
We then describe the action of Lie bracket on this basis.
First, recall the direct sum decomposition of $O$ from \eqref{eq(1):dsO[3012]}.

\begin{lemma}\label{lem:3012,O=d.s}
We have
\begin{equation*}
	O = x\otimes \mathbb{F}[t]  + (x \otimes 1 - z \otimes (t-1) )\mathbb{F}[t] + (y \otimes t + z \otimes (t-1)) \mathbb{F}[t] \qquad ({\rm direct \ sum}).
\end{equation*}
Moreover, 
\begin{equation*}
	X^\sigma_{12} \cap O = x \otimes \mathbb{F}[t], \quad	
	X^\sigma_{01} \cap O = (x \otimes 1 - z \otimes (t-1))\mathbb{F}[t], \quad 
	X^\sigma_{30} \cap O = (y \otimes t + z \otimes (t-1)) \mathbb{F}[t].
\end{equation*}
\end{lemma}
\begin{proof}
Recall the group $G$ from Definition \ref{def:Z2xZ2-group} and the permutation $\rho\tau=(03) \in G$.
Apply $\rho\tau$ to each equation in \eqref{eq(2):0312 O=d.s} in the same manner as in Lemma \ref{lem:3021,O=d.s}.
\end{proof}

Define the elements $A^{\uparrow\downarrow}_i, B^{\uparrow\downarrow}_i, \psi^{\uparrow\downarrow}_i \in \L$ as follows.
\begin{align}
	&& A^{\uparrow\downarrow}_{i} & :=  x \otimes (-t)^i && i\geq 0, && \label{[3012]: A}\\
	&& B^{\uparrow\downarrow}_{i} & :=  (- y \otimes t - z \otimes (t-1)) (-t)^i && i\geq 0, && \label{[3012]: B}\\
	&& \psi^{\uparrow\downarrow}_{i} & :=  (x \otimes 1 - z \otimes (t-1) ) (-t)^{i-1} && i\geq 0. && \label{[3012]: psi}
\end{align}
Observe that 
\begin{equation*}
	A^{\uparrow\downarrow}_0 = x^\sigma_{12}, \qquad 
	B^{\uparrow\downarrow}_{0} = x^\sigma_{30}, \qquad
	\psi^{\uparrow\downarrow}_{0}=x^\sigma_{01}.
\end{equation*}
Recall the standard generators $A$, $B$ of $O$.
Note that $A^{\uparrow\downarrow}_{0} = A$ and $B^{\uparrow\downarrow}_{0} = - B$.
Note also that $\psi^{\uparrow\downarrow}_{0}$ is not in $O$.

\begin{lemma}\label{lem:3012 basis}
Referring to \eqref{[3012]: A}--\eqref{[3012]: psi}, the following {\rm(i)}--{\rm(iii)} hold.
\begin{itemize}
	\item[\rm(i)] The elements $\{A^{\uparrow\downarrow}_{i}\}_{i \in \mathbb{N}}$ form a basis for $X^\sigma_{12} \cap O$.
	\item[\rm(ii)] The elements $\{B^{\uparrow\downarrow}_{i}\}_{i \in \mathbb{N}}$ form a basis for $X^\sigma_{30} \cap O$.
	\item[\rm(iii)] The elements $\{\psi^{\uparrow\downarrow}_{i+1}\}_{i \in \mathbb{N}}$ form a basis for $X^\sigma_{01}\cap O$.
\end{itemize}
\end{lemma}
\begin{proof}
Similar to Lemma \ref{lem:0312 basis}.
\end{proof}

\begin{theorem}\label{thm:3012 basis}
Referring to \eqref{[3012]: A}--\eqref{[3012]: psi}, the Onsager Lie algebra $O$ has a basis
\begin{equation}\label{basis[3012]}
	A^{\uparrow\downarrow}_{i}, \qquad 
	B^{\uparrow\downarrow}_{i}, \qquad 
	\psi^{\uparrow\downarrow}_{i+1}, \qquad \qquad i\in \mathbb{N},
\end{equation}
where $A^{\uparrow\downarrow}_0 = A$, $B^{\uparrow\downarrow}_0 = - B$.
The Lie bracket acts on this basis as follows. 
For $i,j \in \mathbb{N}$,
\begin{align}
	[A^{\uparrow\downarrow}_{i}, A^{\uparrow\downarrow}_{j}] & = 0, \label{thm3012:eq[A,A]}\\
	[B^{\uparrow\downarrow}_{i}, B^{\uparrow\downarrow}_{j}] & = 0, \\
	[\psi^{\uparrow\downarrow}_{i}, \psi^{\uparrow\downarrow}_{j}] & = 0, \\
	[\psi^{\uparrow\downarrow}_{i}, A^{\uparrow\downarrow}_{j}] & = 2\psi^{\uparrow\downarrow}_{i+j} + 2A^{\uparrow\downarrow}_{i+j}, \\
	[B^{\uparrow\downarrow}_{i}, \psi^{\uparrow\downarrow}_{j}] & = 2B^{\uparrow\downarrow}_{i+j} + 2\psi^{\uparrow\downarrow}_{i+j}, \\
	[A^{\uparrow\downarrow}_{i}, B^{\uparrow\downarrow}_{j}] & = 2A^{\uparrow\downarrow}_{i+j} + 2B^{\uparrow\downarrow}_{i+j} - 4\psi^{\uparrow\downarrow}_{i+j+1}. \label{thm3012:eq[A,B]}
\end{align}
\end{theorem}
\begin{proof}
By Lemmas \ref{lem:3012,O=d.s} and \ref{lem:3012 basis}, the elements in \eqref{basis[3012]} form a basis for $O$.
Equations \eqref{thm3012:eq[A,A]}--\eqref{thm3012:eq[A,B]} are obtained in a manner similar to that used in Theorem \ref{thm:0312 basis}.
\end{proof}

\begin{notation}
We denote by $[3012]$ the basis \eqref{basis[3012]} for $O$.
\end{notation}

\begin{corollary}\label{cor:[3012] rec}
The basis $[3012]$ for $O$ is recursively obtained in the following order
$$
	A^{\uparrow\downarrow}_0, \quad 
	B^{\uparrow\downarrow}_0, \quad 
	\psi^{\uparrow\downarrow}_1, \quad 
	A^{\uparrow\downarrow}_1, \quad 
	B^{\uparrow\downarrow}_1, \quad 
	\psi^{\uparrow\downarrow}_2, \quad 
	A^{\uparrow\downarrow}_2, \quad 
	B^{\uparrow\downarrow}_2, \quad 
	\psi^{\uparrow\downarrow}_3, \quad  \ldots
$$
using $A^{\uparrow\downarrow}_0 = A$, $B^{\uparrow\downarrow}_0 = - B$ and the following equations for $i \geq 1$:
\begin{align*}
	\psi^{\uparrow\downarrow}_{i} & = \frac{A^{\uparrow\downarrow}_{i-1}}{2} + \frac{B^{\uparrow\downarrow}_{i-1}}{2} + \frac{[A^{\uparrow\downarrow}_{i-1}, B]}{4}, \\ 
	A^{\uparrow\downarrow}_{i} & = \frac{[\psi^{\uparrow\downarrow}_{i}, A]}{2} - \psi^{\uparrow\downarrow}_{i}, \\ 
	B^{\uparrow\downarrow}_{i} & = - \frac{[B, \psi^{\uparrow\downarrow}_{i}]}{2} - \psi^{\uparrow\downarrow}_{i}. 
\end{align*}
\end{corollary}
\begin{proof}
Similar to Corollary \ref{cor:[0312] rec}.
\end{proof}

\begin{example}
Recall the basis elements of $[3012]$ in \eqref{basis[3012]}.
We express $A^{\uparrow\downarrow}_{i}$, $B^{\uparrow\downarrow}_{i}$, $\psi^{\uparrow\downarrow}_{i+1}$ for $0\leq i \leq 2$.
Observe that
\begin{equation*}
	A^{\uparrow\downarrow}_0 = A, \qquad \qquad B^{\uparrow\downarrow}_0 = -B.
\end{equation*}
By Corollary \ref{cor:[3012] rec}, 
\begin{align*}
	\psi^{\uparrow\downarrow}_{1} & = \frac{A}{2} - \frac{B}{2} + \frac{1}{4}[A,B], \\
	A^{\uparrow\downarrow}_{1} & = - \frac{A}{2} + \frac{B}{2} + \frac{1}{8}[A,[B,A]], \\
	B^{\uparrow\downarrow}_{1} & = -\frac{A}{2} + \frac{B}{2} - \frac{1}{8}[B,[A,B]], \\
	\psi^{\uparrow\downarrow}_{2} & = -\frac{A}{2} + \frac{B}{2} + \frac{1}{8}[B,A] + \frac{1}{16} [A, [B, A]] - \frac{1}{16}[B,[A,B]] - \frac{1}{32} [B, [A, [B, A]]], \\
	A^{\uparrow\downarrow}_{2} & = \frac{A}{2} - \frac{B}{2} - \frac{1}{8}[A, [B, A]] + \frac{1}{16} [B, [A, B]] + \frac{1}{64} [A, [B, [A, [B, A]]]], \\
	B^{\uparrow\downarrow}_{2} & = \frac{A}{2} - \frac{B}{2} + \frac{1}{8}[B, [A, B]] - \frac{1}{16} [A, [B, A]] - \frac{1}{64} [B, [A, [B, [A, B]]]],\\
	\psi^{\uparrow\downarrow}_{3} & = \frac{A}{2} - \frac{B}{2} - \frac{1}{16}[B,A] - \frac{3}{32}[A,[B,A]] + \frac{3}{32}[B,[A,B]] + \frac{1}{32}[B, [A,[B,A]]] \\
	& \qquad + \frac{1}{128} [A,[B,[A,[B,A]]]] - \frac{1}{128} [B,[A,[B,[A,B]]]] - \frac{1}{256}[B,[A,[B,[A,[B,A]]]]].
\end{align*}
\end{example}


\section{Actions of $\rho$, $\tau$ on the four bases}\label{sec:action rho tau}
Recall the permutations $\rho$ and $\tau$ from \eqref{eq:rho,tau}.
In Lemma \ref{lem:action rho, tau}, we saw how $\rho$, $\tau$ act on the Onsager Lie algebra $O$.
Recall the four bases for $O$ defined in Sections \ref{sec:0312}--\ref{sec:3012}.
In this section, we discuss how $\rho$, $\tau$ act on the four bases.

\begin{theorem}\label{thm:4bases,auto}
Let $\rho$, $\tau$ be automorphisms of $\L$ as in Lemma \ref{lem:EldThm1.4}.
The following {\rm(i)}--{\rm(iv)} hold.
\begin{enumerate}[label=(\roman*), font=\normalfont]

	\item The automorphism $\rho$ swaps the bases $[0312]$ and $[3021]$ as follows. 
	$$
	\renewcommand{\arraystretch}{1.2}
	\begin{array}{lcccl}
	\qquad [0312]&&&& \qquad [3021] \\
	\hline
	A^{\uparrow\uparrow}_i = x \otimes (t-1)^i && \longleftrightarrow && A^{\downarrow\downarrow}_i =  - x \otimes (t-1)^i \\
	B^{\uparrow\uparrow}_i = (y \otimes t + z \otimes (t-1)) (t-1)^i && \longleftrightarrow && B^{\downarrow\downarrow}_i = (- y \otimes t - z \otimes (t-1)) (t-1)^i\\
	\psi^{\uparrow\uparrow}_{i+1} = z \otimes (t-1)^{i+1} && \longleftrightarrow && \psi^{\downarrow\downarrow}_{i+1} = (-x\otimes 1 - y \otimes t) (t-1)^{i}
	\end{array}
	$$

	\item The automorphism $\rho$ swaps the bases  $[0321]$ and $[3012]$ as follows. 
	$$
	\renewcommand{\arraystretch}{1.2}
	\begin{array}{lcccl}
	\qquad [0321] &&&& \qquad [3012] \\
	\hline
	A^{\downarrow\uparrow}_i =  - x \otimes (-t)^i && \longleftrightarrow && A^{\uparrow\downarrow}_i = x \otimes (-t)^i \\
	B^{\downarrow\uparrow}_i =  (y \otimes t + z \otimes (t-1)) (-t)^i && \longleftrightarrow && B^{\uparrow\downarrow}_i = (- y \otimes t - z \otimes (t-1)) (-t)^i\\
	\psi^{\downarrow\uparrow}_{i+1} = - y \otimes (-t)^{i+1} && \longleftrightarrow && \psi^{\uparrow\downarrow}_{i+1} = ( x \otimes 1 - z \otimes (t-1) ) (-t)^{i}
	\end{array}
	$$

	\item The automorphism $\tau$ swaps the bases $[0312]$ and $[0321]$ as follows.
	$$
	\renewcommand{\arraystretch}{1.2}
	\begin{array}{lcccl}
	\qquad [0312] &&&& \qquad [0321] \\
	\hline
	A^{\uparrow\uparrow}_i = x \otimes (t-1)^i && \longleftrightarrow && A^{\downarrow\uparrow}_i = - x \otimes (-t)^i \\
	B^{\uparrow\uparrow}_i = (y \otimes t + z \otimes (t-1)) (t-1)^i && \longleftrightarrow && B^{\downarrow\uparrow}_i = (y \otimes t + z \otimes (t-1)) (-t)^i\\
	\psi^{\uparrow\uparrow}_{i+1} = z \otimes (t-1)^{i+1} && \longleftrightarrow && \psi^{\downarrow\uparrow}_{i+1} =  - y \otimes (-t)^{i+1}
	\end{array}
	$$

	\item The automorphism $\tau$ swaps the bases $[3021]$ and $[3012]$ as follows.
	$$
	\renewcommand{\arraystretch}{1.2}
	\begin{array}{lcccl}
	\qquad [3021] &&&& \qquad [3012] \\
	\hline
	A^{\downarrow\downarrow}_i = - x \otimes (t-1)^i && \longleftrightarrow && A^{\uparrow\downarrow}_i = x \otimes (-t)^i \\
	B^{\downarrow\downarrow}_i = (- y \otimes t - z \otimes (t-1)) (t-1)^i && \longleftrightarrow && B^{\uparrow\downarrow}_i = (- y \otimes t - z \otimes (t-1)) (-t)^i\\
	\psi^{\downarrow\downarrow}_{i+1} =  (-x\otimes 1 - y \otimes t) (t-1)^{i} && \longleftrightarrow && \psi^{\uparrow\downarrow}_{i+1} =  ( x \otimes 1 - z \otimes (t-1) ) (-t)^{i}
	\end{array}
	$$

\end{enumerate}
\end{theorem}
\begin{proof}
(i): Note that $\rho^2=1$. 
Apply $\rho$ to the elements $A^{\uparrow\uparrow}_i, B^{\uparrow\uparrow}_i, \psi^{\uparrow\uparrow}_{i+1}$ using Lemma \ref{lem:EldThm1.4}(i), and then simplify the results to obtain $A^{\downarrow\downarrow}_i, B^{\downarrow\downarrow}_i, \psi^{\downarrow\downarrow}_{i+1}$, respectively.\\
(ii)--(iv): Similar.
\end{proof}

Theorem \ref{thm:4bases,auto} is summarized in the following diagram:
\begin{equation}\label{eq:4bases diagram}
    \begin{tikzcd}
    {[0312]} \arrow[r, leftrightarrow, "\rho"] \arrow[d, leftrightarrow, "\tau" swap] & {[3021]} \arrow[d, leftrightarrow, "\tau"] \\
    {[0321]} \arrow[r, leftrightarrow, "\rho"] & {[3012]} \arrow[u, leftrightarrow]
    \end{tikzcd}
\end{equation}

\section{Transition matrices}\label{sec:TM}
Recall the four bases for $O$ shown in \eqref{eq:4bases diagram}.
In this section, we find the transition matrices between each pair of the four bases that are adjacent in the diagram \eqref{eq:4bases diagram}.

\begin{lemma}\label{lem:TM 3021-0312}
Recall the bases $[0312]$ from \eqref{basis:Ai Bi phi} 
and $[3021]$ from \eqref{basis[3021]}.
Then the transition matrix from $[0312]$ to $[3021]$ is given as follows.
\begin{equation}\label{tm:0312->3021}
	A^{\downarrow\downarrow}_i = - A^{\uparrow\uparrow}_i, \qquad 
	B^{\downarrow\downarrow}_i = - B^{\uparrow\uparrow}_i, \qquad
	\psi^{\downarrow\downarrow}_{i+1} = - A^{\uparrow\uparrow}_i - B^{\uparrow\uparrow}_i + \psi^{\uparrow\uparrow}_{i+1} \qquad (i\in \mathbb{N}).
\end{equation}
Moreover, the transition matrix from $[3021]$ to $[0312]$ is given as follows.
\begin{equation}\label{tm:3021->0312}
	A^{\uparrow\uparrow}_i = - A^{\downarrow\downarrow}_i, \qquad 
	B^{\uparrow\uparrow}_i = - B^{\downarrow\downarrow}_i, \qquad
	\psi^{\uparrow\uparrow}_{i+1} =  - A^{\downarrow\downarrow}_i - B^{\downarrow\downarrow}_i + \psi^{\downarrow\downarrow}_{i+1} \qquad (i\in \mathbb{N}).
\end{equation}
\end{lemma}
\begin{proof}
First, we show \eqref{tm:0312->3021}.
Using \eqref{def:A(i)}--\eqref{def:psi(i)} and \eqref{[3021]: A}--\eqref{[3021]: psi}, we have
\begin{align*}
	A^{\downarrow\downarrow}_i  & = -x \otimes (t-1)^i = - A^{\uparrow\uparrow}_i,\\
	B^{\downarrow\downarrow}_i  & = - (y \otimes t + z \otimes (t-1))(t-1)^i = - B^{\uparrow\uparrow}_i,\\
	\psi^{\downarrow\downarrow}_{i+1} & = -x\otimes (t-1)^i - y\otimes t(t-1)^i \\
	& = -x\otimes (t-1)^i - (y\otimes t + z\otimes (t-1))(t-1)^i + z\otimes(t-1)^{i+1}\\
	& = - A^{\uparrow\uparrow}_i - B^{\uparrow\uparrow}_i + \psi^{\uparrow\uparrow}_{i+1},
\end{align*}
for $i\in \mathbb{N}$. 
We have shown \eqref{tm:0312->3021}.
Similarly, we obtain \eqref{tm:3021->0312}. 
\end{proof}

\begin{lemma}\label{lem:TM 0321-3012}
Recall the bases $[0321]$ from \eqref{basis[0321]} and $[3012]$ from \eqref{basis[3012]}.
Then the transition matrix from $[0321]$ to $[3012]$ is given as follows.
\begin{equation*}
	A^{\uparrow\downarrow}_i = - A^{\downarrow\uparrow}_i, \qquad 
	B^{\uparrow\downarrow}_i = - B^{\downarrow\uparrow}_i, \qquad
	\psi^{\uparrow\downarrow}_{i+1} = - A^{\downarrow\uparrow}_i - B^{\downarrow\uparrow}_i + \psi^{\downarrow\uparrow}_{i+1} \qquad (i\in \mathbb{N}).
\end{equation*}
Moreover, the transition matrix from $[3012]$ to $[0321]$ is given as follows.
\begin{equation*}
	A^{\downarrow\uparrow}_i = - A^{\uparrow\downarrow}_i, \qquad 
	B^{\downarrow\uparrow}_i = - B^{\uparrow\downarrow}_i, \qquad
	\psi^{\downarrow\uparrow}_{i+1} =  - A^{\uparrow\downarrow}_i - B^{\uparrow\downarrow}_i + \psi^{\uparrow\downarrow}_{i+1} \qquad (i\in \mathbb{N}).
\end{equation*}
\end{lemma}
\begin{proof}
Similar to Lemma \ref{lem:TM 3021-0312}.
\end{proof}

\begin{lemma}\label{lem:TM 0312-0321}
Recall the bases 
$[0312]$ from \eqref{basis:Ai Bi phi} 
and 
$[0321]$ from \eqref{basis[0321]}. 
Then the transition matrix from $[0312]$ to $[0321]$ is given as follows.
\begin{align}
	& A^{\downarrow\uparrow}_{i}  = (-1)^{i+1} \sum^i_{j=0} {\binom{i}{j}}A^{\uparrow\uparrow}_{j}, \label{tm eq(1):0312->0321}\\
	& B^{\downarrow\uparrow}_{i}  = (-1)^i \sum^i_{j=0} {\binom{i}{j}}B^{\uparrow\uparrow}_{j}, \label{tm eq(2):0312->0321}\\
	& \psi^{\downarrow\uparrow}_{i+1}  = (-1)^i \sum^i_{j=0} {\binom{i}{j}}B^{\uparrow\uparrow}_{j} + (-1)^{i+1} \sum^i_{j=0} {\binom{i}{j}} \psi^{\uparrow\uparrow}_{j+1} &&(i\in \mathbb{N}).\label{tm eq(3):0312->0321}
\end{align}
Moreover, the transition matrix from $[0321]$ to $[0312]$ is given as follows.
\begin{align}
	& A^{\uparrow\uparrow}_{i}  = (-1)^{i+1} \sum^i_{j=0} {\binom{i}{j}}A^{\downarrow\uparrow}_{j}, \label{tm eq(1):0321->0312} \\
	& B^{\uparrow\uparrow}_{i}  = (-1)^i \sum^i_{j=0} {\binom{i}{j}}B^{\downarrow\uparrow}_{j}, \label{tm eq(2):0321->0312} \\
	& \psi^{\uparrow\uparrow}_{i+1}  = (-1)^i \sum^i_{j=0} {\binom{i}{j}}B^{\downarrow\uparrow}_{j} + (-1)^{i+1} \sum^i_{j=0} {\binom{i}{j}} \psi^{\downarrow\uparrow}_{j+1} &&(i\in \mathbb{N}).\label{tm eq(3):0321->0312}
\end{align}
\end{lemma}
\begin{proof}
First we show \eqref{tm eq(1):0312->0321}.
We have
\begin{align*}
	&&A^{\downarrow\uparrow}_{i} & = -x \otimes (-t)^i && \text{(by \eqref{[0321]: A})}&&\\
	&&& = - x \otimes (-1)^i(t-1+1)^i &&\\
	&&& = (-1)^{i+1} \sum^i_{j=0} {\binom{i}{j}} x\otimes (t-1)^j && \text{(by the binomial theorem)} && \\
	&&& = (-1)^{i+1} \sum^i_{j=0} {\binom{i}{j}}A^{\uparrow\uparrow}_{j} && \text{(by \eqref{def:A(i)})}. &&
\end{align*}
We have shown \eqref{tm eq(1):0312->0321}.
Similarly, use \eqref{def:A(i)}--\eqref{def:psi(i)} and \eqref{[0321]: A}--\eqref{[0321]: psi} to get \eqref{tm eq(2):0312->0321}--\eqref{tm eq(3):0321->0312}.
\end{proof}

\begin{lemma}\label{lem:TM 3021-3012}
Recall the basis $[3021]$ from \eqref{basis[3021]} and $[3012]$ from \eqref{basis[3012]}. 
Then the transition matrix from $[3021]$ to $[3012]$ is given as follows.
\begin{align*}
	& A^{\uparrow\downarrow}_{i}  = (-1)^{i+1} \sum^i_{j=0} {\binom{i}{j}}A^{\downarrow\downarrow}_{j}, \\
	& B^{\uparrow\downarrow}_{i}  = (-1)^i \sum^i_{j=0} {\binom{i}{j}}B^{\downarrow\downarrow}_{j}, \\
	& \psi^{\uparrow\downarrow}_{i+1}  = (-1)^i \sum^i_{j=0} {\binom{i}{j}}B^{\downarrow\downarrow}_{j} + (-1)^{i+1} \sum^i_{j=0} {\binom{i}{j}} \psi^{\downarrow\downarrow}_{j+1} &&(i\in \mathbb{N}).
\end{align*}
Moreover, the transition matrix from $[3012]$ to $[3021]$ is given as follows.
\begin{align*}
	& A^{\downarrow\downarrow}_{i}  = (-1)^{i+1} \sum^i_{j=0} {\binom{i}{j}}A^{\uparrow\downarrow}_{j}, \\
	& B^{\downarrow\downarrow}_{i}  = (-1)^i \sum^i_{j=0} {\binom{i}{j}}B^{\uparrow\downarrow}_{j}, \\
	& \psi^{\downarrow\downarrow}_{i+1}  = (-1)^i \sum^i_{j=0} {\binom{i}{j}}B^{\uparrow\downarrow}_{j} + (-1)^{i+1} \sum^i_{j=0} {\binom{i}{j}} \psi^{\uparrow\downarrow}_{j+1} &&(i\in \mathbb{N}).
\end{align*}
\end{lemma}
\begin{proof}
Similar to Lemma \ref{lem:TM 0312-0321}.
\end{proof}

\appendix
\section{Appendix}\label{sec:Appen}


In Sections \ref{sec:0312}--\ref{sec:3012}, we have displayed four bases for $O$: $[0312]$, $[3021]$, $[0321]$, and $[3012]$.
We summarize these bases in the following table.
\begin{equation*}
\begingroup
\renewcommand{\arraystretch}{1.3}
	\begin{tabular}{c|clc|cc}
	basis name  & & \qquad basis elements $(i\in \mathbb{N})$ & & basis for\\
	\hline \hline
	\multirow{3}{*}{[0312]} 
		& & $B^{\uparrow\uparrow}_{i} =  (y \otimes t + z \otimes (t-1)) (t-1)^i$ & & $X^\sigma_{03} \cap O$\\
		& & $\psi^{\uparrow\uparrow}_{i+1} =  z \otimes (t-1)^{i+1} $ & & $X^\sigma_{31} \cap O$ \\ 
		& & $A^{\uparrow\uparrow}_{i} =  x \otimes (t-1)^i$ & & $X^\sigma_{12} \cap O$\\ \hline
	
	\multirow{3}{*}{[3021]}
		& & $B^{\downarrow\downarrow}_{i} =  (- y \otimes t - z \otimes (t-1)) (t-1)^i $ & & $X^\sigma_{30} \cap O$\\	
		& & $\psi^{\downarrow\downarrow}_{i+1} =  (-x\otimes 1 - y \otimes t) (t-1)^{i}$ & & $X^\sigma_{02} \cap O$\\
		& & $A^{\downarrow\downarrow}_{i} =  - x \otimes (t-1)^i $ & & $X^\sigma_{21}  \cap O $\\ \hline	
	\multirow{3}{*}{[0321]}
		& & $B^{\downarrow\uparrow}_{i} :=  (y \otimes t + z \otimes (t-1)) (-t)^i$ & & $X^\sigma_{03} \cap O$\\
		& & $\psi^{\downarrow\uparrow}_{i+1} :=  - y \otimes (-t)^{i+1}$ & & $X^\sigma_{32} \cap O $\\
		& & $A^{\downarrow\uparrow}_{i} =  - x \otimes (-t)^i$ & & $X^\sigma_{21} \cap O$ \\ \hline	
	
	\multirow{3}{*}{[3012]} 
		& & $B^{\uparrow\downarrow}_{i} =  (- y \otimes t - z \otimes (t-1)) (-t)^i$ & & $X^\sigma_{30} \cap O$\\	
		& & $\psi^{\uparrow\downarrow}_{i+1} =  ( x \otimes 1 - z \otimes (t-1) ) (-t)^{i}$ & & $X^\sigma_{01} \cap O$\\	
		& & $A^{\uparrow\downarrow}_{i} =  x \otimes (-t)^i $ & & $X^\sigma_{12} \cap O$
	\end{tabular}
\endgroup
\end{equation*}
Moreover, the standard generators $A$, $B$ of $O$ are expressed in the following basis elements.
\begin{equation*}
\begingroup
\renewcommand{\arraystretch}{1.3}
\begin{tabular}{>{\centering\arraybackslash}p{1.5cm}||>{\centering\arraybackslash}p{2cm}|>{\centering\arraybackslash}p{2cm}|>{\centering\arraybackslash}p{2cm}|>{\centering\arraybackslash}p{2cm}}
	basis & $[0312]$ & $[3021]$ & $[0321]$ & $[3012]$ \\
	\hline
	$A$ & $A^{\uparrow\uparrow}_0$ & $-A^{\downarrow\downarrow}_0$ & $-A^{\downarrow\uparrow}_0$ & $A^{\uparrow\downarrow}_0$ \\
	$B$ & $B^{\uparrow\uparrow}_0$ & $-B^{\downarrow\downarrow}_0$ & $B^{\downarrow\uparrow}_0$ & $-B^{\uparrow\downarrow}_0$
\end{tabular}
\endgroup
\end{equation*}

\section*{Acknowledgement}
The author expresses his deepest gratitude to Paul Terwilliger for many valuable conversations about this paper. This work was partially completed during the author’s sabbatical visit to Pohang University of Science and Technology (POSTECH) in Korea, from August 15, 2024, to January 15, 2025. The author is grateful to his host, Jae-Hun Jung, and the members of the Mathematics Department at POSTECH for their warm hospitality and generous support.





\end{document}